\documentclass[11pt]{amsart}

\usepackage{amssymb,verbatim}
\usepackage{amsmath,amsfonts}
\usepackage[mathscr]{euscript} 
\usepackage{amsthm}
\usepackage{url}
\usepackage{graphicx} 
\usepackage{enumitem} 
\usepackage{hyperref} 
\hypersetup{colorlinks} 
\usepackage{bbm} 
\usepackage{bm} 
\usepackage{cancel} 

\usepackage[colorinlistoftodos]{todonotes}

\RequirePackage{cleveref}
\usepackage{hypcap}
\hypersetup{colorlinks=true, citecolor=darkblue, linkcolor=darkblue}
\definecolor{darkblue}{rgb}{0.0,0,0.7}
\newcommand{\darkblue}{\color{darkblue}}

\definecolor{darkred}{rgb}{0.68,0,0}

\definecolor{darkgreen}{rgb}{0,.38,0}

\newcommand{\defn}[1]{\emph{\darkblue #1}}

\newcommand{\defnb}[1]{\emph{\darkblue #1}}

\setlist[enumerate]{
	label=\textnormal{({\roman*})},
	ref={\roman*}}

\makeatletter
\def\th@plain{%
	\thm@notefont{}
	\itshape 
}
\def\th@definition{%
	\thm@notefont{}
	\normalfont 
}
\makeatother

\newtheorem{thm}{Theorem}[section]
\newtheorem{lemma}[thm]{Lemma}

\newtheorem*{claim*}{Claim}

\newtheorem{prop}[thm]{Proposition}
\newtheorem{conj}[thm]{Conjecture}
\newtheorem{conv}[thm]{Convention}

\theoremstyle{definition}

\newtheorem{rem}[thm]{Remark}

\numberwithin{figure}{section}
\numberwithin{equation}{section}

\def\zz{\mathbb Z}
\def\nn{\mathbb N}

\def\rr{\mathbb R}

\def\Ga{\Gamma}

\def\ga{\gamma}

\def\de{\delta}

\def\al{\alpha}
\def\be{\beta}

\def\cA{\mathcal A}

\def\cE{\mathcal E}

\def\cF{\mathcal F}

\def\cN{\mathcal N}

\def\cO{\mathcal O}

\def\cR{\mathcal R}

\def\cT{\mathcal T}

\def\<{\langle}
\def\>{\rangle}

\def\Z{\zz}

\def\SL{ {\text {\rm SL} } }

\def\Ups{{\small {\Upsilon}}}

\def\vt{\vartheta}

\def\0{{\mathbf 0}}

\def\.{\hskip.06cm}
\def\ts{\hskip.03cm}

\def\bv{\textbf{\textit{v}}}

\def\fC{\frak{C}}

\def\fF{\frak{F}}

\def\fS{\frak{S}}

\def\.{\hskip.06cm}
\def\ts{\hskip.03cm}

\def\G{\Ga}

\newenvironment{psmallmatrix}
  {\left(\begin{smallmatrix}}
  {\end{smallmatrix}\right)}

\def\cM{\mathcal M}

\DeclareMathOperator{\at}{\tau}

\DeclareMathOperator{\f}{\Phi}
\DeclareMathOperator{\g}{\Psi}
\DeclareMathOperator{\h}{\Ups}

\newcommand{\kk}{t}
\newcommand{\jj}{u}

\newcommand{\Ais}{110}
\newcommand{\Nis}{5}

\newcommand{\mattwos}[4]
{\bigl( \begin{smallmatrix}
                        #1  & #2   \\
                        #3 &  #4
\end{smallmatrix} \bigr)
}

\newcommand{\mattwo}[4]
{\left(\begin{array}{cc}
                        #1  & #2   \\
                        #3 &  #4
                          \end{array}\right) }

\def\R{\mathbb R}
\def\sL{\mathscr L}

\usetikzlibrary{positioning,shapes.geometric,calc}

\title[Spanning trees and continued fractions]
{Spanning trees and continued fractions}
\date{\today}

 \author{Swee Hong Chan}
 \address[Swee Hong Chan]{Department of Mathematics, Rutgers University,  Piscataway, NJ 08854.}
 \email{\texttt{sweehong.chan@rutgers.edu}}

 \author{Alex Kontorovich}
 \address[Alex Kontorovich]{Department of Mathematics, Rutgers University,  Piscataway, NJ 08854.}
 \email{\texttt{alex.kontorovich@rutgers.edu}}

 \author[\ts Igor Pak]{Igor Pak}
 \address[Igor Pak]{Department of Mathematics, UCLA,  Los Angeles, CA 90095.}
 \email{\texttt{pak@math.ucla.edu}}

\begin{document}

\begin{abstract}
We prove
the exponential growth of the cardinality of  the set of numbers of spanning trees in simple (and planar) graphs on $n$ vertices, answering a question of  Sedl\'a\v{c}ek from 1969.
The proof uses a connection with continued fractions,
``thin orbits,'' and
Zaremba's conjecture.
\end{abstract}
	
\maketitle

\section{Introduction}\label{s:intro}

\subsection{Main results} \label{ss:intro-main}
One of the most basic invariants of a graph \ts $G=(V,E)$, \ts is the number,
denoted $\tau(G)$, of spanning trees in~$G$.
This fascinating quantity measures a kind of ``complexity'' of $G$, and
appears in many different contexts across mathematical sciences:
from Commutative Algebra to Probability, from Lie Theory to
Combinatorial Optimization, etc.

The study of the set of \emph{spanning tree numbers} $\tau(G)$ for $G$ ranging in various families 
of simple graphs  graded by the number 
of vertices of $G$, was initiated almost 60 years ago in a series of papers by Sedl\'a\v{c}ek  \cite{Sed66, Sed69, Sed70}, who highlighted in particular the families: $(i)$ of all simple graphs, $(ii)$ $k$-regular graphs, and $(iii)$ planar graphs. In this paper, we focus on the latter; see \S\ref{ss:finrem-graphs} and \S\ref{sec:Spec} for some discussion on the others.

Recall that graph is called \emph{simple} if it has no loops or multiple edges.
For \ts $n\ge3$, let $\cT(n)$ denote the
 set of numbers of
spanning trees of connected planar simple graphs on~$n$ vertices:
\begin{equation}\label{eq:def-alpha}
\cT(n) \, := \, \big\{\ts \tau(G) \, : \, G=(V,E) \ \text{is
connected,
simple, and planar, and} \ |V| = n\ts\big\}.
\end{equation}
For example,   
\. $\cT(4)=\{1,3,4,8,16\}$ \. 
and \. $\cT(5) = \{1,3,4,5,8,9,11,12,16,20,21,24,40,45,75\}$.
It is easy to see that the sets \ts $\cT(n)$ \ts are nested, i.e.\
\ts $\cT(n) \subseteq \cT(n+1)$, that \ts $n \in \cT(n)$, and that
\ts $2\notin \cT(n)$ \ts for all~$n$.

Let \ts $|\cT(n)|$ \ts be the cardinality of \ts $\cT(n)$,
%
%
 that is,
the number of distinct values of \ts $\tau(G)$ \ts in \ts $\cT(n)$;
so \ts $|\cT(4)|=5$ \ts and \ts $|\cT(5)|=15$.
A simple argument using Euler characteristic (see \S\ref{ss:exp-upper-bnd}), shows that
\begin{equation}\label{eq:Tn8n}
\max\cT(n)<C^{n},
\end{equation}
for  some \ts $C>1$. Hence \ts
 $|\cT(n)|$ \ts grows at most exponentially:
$$
|\cT(n)|<C^{n},
$$
for all \ts $n$ \ts sufficiently large.
The first lower bound was given already by Sedl\'a\v{c}ek himself in 1969,
who
proved
that \ts $|\cT(n)| = \Omega(n^2)$, see  \cite{Sed69}.  In \cite{Aza14}, Azarija showed that
\ts $|\cT(n)| = e^{\Omega(\sqrt{n/\log n})}$.
It follows from Stong's theorem (Theorem~\ref{t:Stong}) that \ts
$|\cT(n)| = e^{\Omega(n^{2/3})}$.
Our
main
result shows that the sequence $|\cT(n)|$
indeed
has exponential growth:

\begin{thm}[{\rm Main Theorem}{}]\label{t:main-exp} \. There exists a constant \. $c>1
$ \. such that
\begin{equation}\label{eq:main-exp-const}
|\cT(n)| \,
> \, c^{n}
\end{equation}
holds for all sufficiently large $n$.
\end{thm}

\begin{rem}
Even for the family of {\it all} simple (not necessarily planar) graphs, \eqref{eq:main-exp-const} proves exponential growth for the first time.
The best previous lower bound was Stong's bound above.
\end{rem}

See \S\ref{ss:finrem-limit} for an explicit 
estimate for
the value \ts $c\approx 1.1103$ \ts derived from our proof method; we made no effort to optimize this constant.
To understand why the lower bound remained out of reach until now, note that
there are only exponentially many nonisomorphic connected simple planar graphs to
begin with,\footnote{While the exact asymptotics remain open, the number of unlabeled
simple planar graphs on $n$ vertices is \ts $O(30.061^n)$, see \cite{B+06},
\cite[$\S$6.9.2]{Noy15} and \cite[\href{https://oeis.org/A003094}{A003094}]{OEIS}.
}
and exponentially many of them may have the same number of spanning trees.\footnote{For
example, there are \ts $\Omega(2.955^n)$ \ts unlabeled trees on~$n$ vertices, see e.g.\
\cite[\href{https://oeis.org/A000055}{A000055}]{OEIS} and \cite[\href{https://oeis.org/A000055}{A051491}]{OEIS}.}

In fact, we are able to prove an even stronger theorem, namely
that the set \ts $\cT(n)$ \ts contains a positive proportion of an exponentially long
set of
integers. 


\begin{thm}
\label{t:main}
There exists a constant \ts $c>1$, such that
\begin{equation}\label{eq:main-thm}
\liminf_{n\to\infty}\frac{1}{c^n} \, \big|\cT(n) \cap \{1,\ldots,c^n\}\big|  \. > \. 0. 
\end{equation}
\end{thm}

Moreover,
it is possible to
improve Theorem~\ref{t:main} from positive proportion to density one, that is,
that the expression in \eqref{eq:main-thm} tends to $1$; see Remark~\ref{rmk:Huang}.
For connections to spectra of operators on locally homogeneous spaces, see \S\ref{sec:Spec}.

%

\smallskip

\subsection{Ingredients} \label{ss:intro-meth}
As we describe in more detail below, there are four main steps in the proofs of Theorems~\ref{t:main-exp} and~\ref{t:main}.
\begin{enumerate}
\item The first is to dualize (cf.~$\S$\ref{ss:intro-back}), replacing the study of $\cT(n)$ with the closely related quantity $\alpha(\kk)$ defined in \eqref{eq:alKdef}.
\item
The second step is to employ a mechanism that constructs simple planar graphs $G=(V,E)$ with a prescribed number of spanning trees, $\tau(G)$, in such a way as to simultaneously allow some control on the number $|V|$ of vertices. It turns out (cf. \S\ref{ss:intro-CF}) that this control relies on certain Diophantine-theoretic properties of continued fraction expansions of rational numbers.
This reduces the graph theory problem to one in Diophantine geometry.
\item
The Diophantine problem turns out to be amenable to
 techniques similar to those attacking Zaremba's conjecture (cf. \S\ref{ss:intro-Zaremba}); one is then able to produce a large enough collection of the desired fractions, as long as a certain Cantor-like fractal's Hausdorff dimension is sufficiently large.
 \item
 And the final ingredient is to verify numerically and rigorously (cf. \S\ref{ss:intro-H-dim})  that this fractal dimension is indeed large enough.
\end{enumerate}

Note that
steps (i) and (ii) already suffice to give an elementary proof of  Theorem~\ref{t:main-exp} (see \S\ref{sec:Pf1p1}).\footnote{See 
also a follow-up paper \cite{ABG} which uses a closely related approach to give an even simpler proof of Theorem~\ref{t:main-exp}.  
The tools in \cite{ABG} do not imply our main result Theorem~\ref{t:main}.}
For Theorem~\ref{t:main}, it is in step (iv) that we turn out to be extraordinarily lucky: the currently best available sufficiency condition is indeed satisfied by the numerics, but only just barely, in the hundredths place! (See Remark~\ref{rmk:vtInfty}.)

\smallskip

\subsection{Dualizing} \label{ss:intro-back}
To better understand Theorem~\ref{t:main-exp},
consider the sets dual to  $\cT(n)$, namely, for $\kk\ge3$, we define $\widehat\cT(\kk):=\{n \ : \ \kk \in \cT(n)\}$. In light of the nesting property $\cT(n)\subseteq\cT(n+1)$, the only quantity of interest for $\widehat\cT(\kk)$ is its minimal element, which we denote by
\begin{equation}\label{eq:alKdef}
\al(\kk)\ := \ \min
\widehat\cT(\kk)
\ = \ \min \{n \ : \ \kk \in \cT(n)\}.
\end{equation}
That is, $\al(\kk)$ is
the smallest number of vertices
of a planar simple graph with exactly $\kk$ spanning trees.
Then $\widehat\cT(\kk)$ consists of $\al(\kk)$, followed by every subsequent integer.
By
\eqref{eq:Tn8n},
we have that
\begin{equation}\label{eq:alOmega}
\al(\kk) = \Omega(\log \kk).
\end{equation}


The study of $\al(\kk)$ was also initiated by Sedl\'{a}\v{c}ek \cite{Sed70} and continued
over the years, see \cite{AS13, Aza14, CP-CF} and~$\S$\ref{ss:finrem-graphs}.  Until Stong's recent
breakthrough, even \ts $\al(\kk) = o(\kk)$ \ts remained open, see \cite[Question~1]{AS13}.

\begin{thm}[{\rm Stong \cite[Cor.~7.3.1]{Stong}}{}] \label{t:Stong}
$\al(\kk) = O\big((\log \kk)^{3/2}/(\log\log \kk)\big)$.
\end{thm}

It is natural to conjecture (see also Proposition \ref{p:Zaremba-strong}),
that the true upper bound for $\al(\kk)$
matches the lower bound in \eqref{eq:alOmega}:

\begin{conj} \label{conj:main}
$\al(\kk) = O(\log \kk)$ \. for all \. $\kk \ge 3$.
\end{conj}

Note that \ts $\al(pq) \le \al(p) + \al(q)$ (by taking disjoint union of graphs), so it suffices to
prove the conjecture for prime~$\kk$.
For Theorem~\ref{t:main-exp}, the following weak progress suffices:\footnote{We thank
Dmitry Krachun for pointing this out.}
\begin{thm}\label{thm:main1p1}
The set of \ts $\kk$ \ts for which Conjecture~\ref{conj:main} holds grows at least like a power. That is, there exist $c, C >0$ so that
\. $
\#\{t<T \ : \ \alpha(t) \, < \, C\log t\} \,  > \, T^c.
$
\end{thm}

The stronger
Theorem~\ref{t:main} is dual to
 the following.

\begin{thm}[{\rm =
Theorem~\ref{t:main}}{}]\label{thm:main2}
Conjecture~\ref{conj:main} holds for a positive proportion of~$\kk$.
That is, there is a set of positive proportion within the natural numbers for which the estimate
\ts $\al(\kk)=O(\log \kk)$ \ts holds
as \ts $\kk\to\infty$ \ts within this set.
\end{thm}

Conjecture~\ref{conj:main} remains out of reach. In fact,
even the following weaker problem is open.
%
Denote by \ts $\be(\kk)$ \ts the smallest number of edges of a planar (not necessarily simple)
graph with exactly $\kk$ spanning trees.  This function was introduced
by Nebesk\'y in \cite{Neb}.
Note that multiple edges are
allowed in this case, and that we have \ts $\be(\kk) < 3 \ts \al(\kk)$.
It was proved in \cite{CP-SY}, that \ts $\be(\kk) = O(\log \kk \. \log\log \kk)$,
see~$\S$\ref{ss:finrem-ave}.  This is very close but still shy of the
natural upper bound that would follow from Conjecture~\ref{conj:main}:

\begin{conj} \label{conj:main-beta}
$\be(\kk) = O(\log \kk)$.
\end{conj}

\smallskip

\subsection{Continued fractions and graphs} \label{ss:intro-CF}
Given \. $a_0\geq 0$, \. $a_1, \ldots, a_\ell \geq 1 $, where  \ts $\ell \geq 0$,
the corresponding \defnb{continued fraction} \ts is defined as follows:
\[ [a_0\ts ; \ts a_1,\ldots, a_\ell] \ := \ a_0  \. +  \.  \cfrac{1}{a_1  \. +  \. \cfrac{1}{\ddots  \ +  \.  \frac{1}{a_\ell}}}\,.
\]
Integers \ts $a_i$ \ts are called \defn{partial quotients}, see e.g.\ \cite[$\S$10.1]{HardyWright}.
We use the notation \ts $[a_1,\ldots, a_\ell]$ \ts when \ts $a_0=0$.
The following elementary result, proved in \S\ref{s:CF},
gives a connection between spanning trees and
continued fractions:

\begin{figure}
\begin{center}
{\small
\begin{tikzpicture}[
    vertex/.style={circle, fill, inner sep=1.5pt},
    node distance=0.5cm
]
    \node[vertex] (a1) {};
    \node[vertex, below=2cm of a1] (a2) {};
    \draw (a1) -- (a2);

    \node[vertex, right=2cm of a1] (b1) {};
    \node[vertex, below=0.5cm of b1] (b2) {};
    \node[vertex, below=0.5cm of b2] (b3) {};
    \draw[dotted] (b3) -- +(0,-0.5cm);
    \node[vertex, below=.7cm of b3] (b4) {};
    \draw (b1) -- (b2) -- (b3);
    \draw (b3) -- (b4);

    \node[vertex, right=2cm of b1] (c1) {};
    \node[vertex, below=0.5cm of c1] (c2) {};
    \node[vertex, below=0.5cm of c2] (c3) {};
    \draw[dotted] (c3) -- +(0,-0.5cm);
    \node[vertex, below=.7cm of c3] (c4) {};
    \draw (c1) -- (c2) -- (c3);
    \draw (c3) -- (c4);

    \node[vertex, right=2cm of c1] (d1) {};
    \node[vertex, below=0.5cm of d1] (d2) {};
    \node[vertex, below=0.5cm of d2] (d3) {};
    \draw[dotted] (d3) -- +(0,-0.5cm);
    \node[vertex, below=.7cm of d3] (d4) {};
    \draw (d1) -- (d2) -- (d3);
    \draw (d3) -- (d4);

    \draw[dashed] ($(b1)-(0,1.1)$) ellipse (0.6cm and 1.4cm);
    \draw[dashed] ($(c1)+(0,-1.1)$) ellipse (0.6cm and 1.4cm);
    \draw[dashed] ($(d1)+(0,-1.1)$) ellipse (0.6cm and 1.4cm);

    \draw[bend left=30] (a1) to (b1);
    \draw[bend left=30] (a1) to (c1);
    \draw[bend left=30] (a1) to (d1);

    \draw (a2) -- (b4);
    \draw (b1) -- (c4);
    \draw (c1) -- ($(c1) !0.33! (d4)$);
    \draw ($(c1) !0.67! (d4)$) -- (d4);

    \node[right=.5cm of c1] {$\cdots$};
    \node at ($(c1)!0.5!(d4)$) {\(\ddots\)};
    \node[above=.1cm of b4] {$\vdots$};
    \node[above=.1cm of c4] {$\vdots$};
    \node[above=.1cm of d4] {$\vdots$};

    \node[below=0.3cm of b4] {$b_m$};
    \node[below=0.3cm of c4] {$b_{m-1}$};
    \node[below=0.3cm of d4] {$b_2$};
\end{tikzpicture}
}\end{center}
\vskip-.3cm
\caption{A typical graph constructed in the Main Graph Theorem~\ref{thm:graphMain}}
\label{fig:1}
\end{figure}

\begin{thm}[{\rm Main Graph Theorem}{}] \label{thm:graphMain}
	Let \ts $\kk, \jj \geq 1$ \ts  be positive integers with \ts $\kk < \jj$ \ts and  \ts $\gcd(\kk,\jj)=1$.
	Suppose that
\begin{equation}\label{eq:graphMain}
 \frac{\kk}{\jj} \ = \   [b_1,1,b_2,1,\ldots, b_m,1],
 \end{equation}
	for some  
	$b_1,\ldots, b_m\ge 1$.
Then there exists a simple planar graph \ts $G$ such that
\[
\tau(G) \. = \. \kk, \quad \text{and} \quad |V| \. = \.  b_2+\ldots+b_m+2\ts.
\]
\end{thm}

The construction, explained in \S\ref{s:CF}, is completely explicit; see Figure~\ref{fig:1} for the typical structure of such a graph.
The proof is inductive, and is a variation on the proof of Theorem 1.5 in \cite{CP-SY} and a
construction of Bier in \cite{Bier}.
Motivated by this theorem, we formulate the following.

\begin{conj}[{\rm Diophantine Conjecture}{}]\label{conj:Zaremba-strong}
There is a universal constant $A>0$ so that,
for every integer \ts $\kk\ge 3$, there is a coprime integer \ts $\jj > \kk$,
such that
the quotient $\kk/\jj$
has continued fraction expansion
 \eqref{eq:graphMain}
 with 
 \. $b_1,\ldots,b_m\le A$.
%
\end{conj}

We discuss in \S\ref{ss:finrem-asy} why this conjecture may be plausible. Regardless, together with the Main Graph Theorem~\ref{thm:graphMain}, it would settle Conjecture~\ref{conj:main}.

\begin{prop}\label{p:Zaremba-strong}
The Diophantine Conjecture~\ref{conj:Zaremba-strong} implies Conjecture~\ref{conj:main}.
\end{prop}

\begin{proof}
For any integer~$\kk\ge3$, Conjecture~\ref{conj:Zaremba-strong} produces a fraction \ts $
\kk/\jj$ \ts
having
continued
fraction expansion as in
\eqref{eq:graphMain}.  By Theorem~\ref{thm:graphMain}, 
there is a simple planar graph $G$ 
such that
$t=\tau(G)$.
This graph also has
\ts $|V| \le \tfrac{(A+1)}{2} (m-1) + 2 = O(\log \kk)$
vertices; hence $\al(\kk)=O(\log \kk)$,
as required.
\end{proof}

While we are not able to establish Conjecture~\ref{conj:Zaremba-strong}, we prove the following approximation.

\begin{thm}\label{thm:BK-strong}
Conjecture~\ref{conj:Zaremba-strong} holds for a set of positive proportion.
That is,
there is some $A>0$, so that the set of $\kk$ for which the conclusion of Conjecture~\ref{conj:Zaremba-strong} holds, has positive proportion in the natural numbers.
\end{thm}

We have thus reduced Theorem~\ref{t:main} to proving  Theorems \ref{thm:graphMain} and \ref{thm:BK-strong}.

\begin{proof}[Proof of Theorem~\ref{thm:main2}]
The proof follows verbatim the proof of Proposition~\ref{p:Zaremba-strong},
but applied to a
positive proportion
subset of $\kk$.
\end{proof}


\smallskip

\subsection{Zaremba's conjecture} \label{ss:intro-Zaremba}
%
The Diophantine Conjecture~\ref{conj:Zaremba-strong} is itself closely related to
 the following celebrated open problem (see, e.g., \cite{Kon13} and $\S$\ref{ss:finrem-Zaremba}
 for more background):\footnote{One think of the Diophantine Conjecture~\ref{conj:Zaremba-strong}
 as a version of Zaremba's Conjecture~\ref{conj:Zaremba} for the \emph{negative continued fractions},
see e.g.\ \cite[$\S$4.4]{BPSZ14}. We will not use this connection and omit the details.  }

\begin{conj}[{\rm \defn{Zaremba's Conjecture} \cite[p.~76]{Zar72}}{}] \label{conj:Zaremba}
There is a universal constant $A>0$ so that,
for every integer \ts $\jj\ge 1$, there is a coprime integer \ts $1\le \kk < \jj$,
such that \. $\kk/\jj= [a_1,\ldots,a_\ell]$ \. and \. $a_1,\ldots,a_\ell\le A$.
%
\end{conj}

Here the change to denominators from numerators in the Diophantine
Conjecture~\ref{conj:Zaremba-strong} is completely anodyne, see Remark~\ref{rmk:NumDenom}.
Zaremba himself conjectured that one can take $A=5$ in Conjecture~\ref{conj:Zaremba},
see also~$\S$\ref{ss:finrem-Zaremba}.  In \cite{BK14}, Bourgain--Kontorovich proved
the following result, on which the proof of Theorem~\ref{thm:BK-strong} is based,
and which shows that Zaremba's conjecture holds for a positive proportion of
denominators (in fact, they prove a density one version, but again,
we will only quote the weaker statement):

\begin{thm}[{\rm \cite{BK14}}{}]\label{t:BK}
There is a universal constant $A>0$ so that the conclusion of Zarembra's Conjecture~\ref{conj:Zaremba} holds for a positive proportion of $\jj$.
\end{thm}

A key role is played here by the Cantor-like sets comprising the limit points of these Diophantine fractions.
In the setting of Zaremba's conjecture, the limit set, for a given $A>0$, is the following:
$$
\fF_A \, := \, \big\{ [a_1,a_2,\ldots] \ : \ 1\le a_i \le A, \ \ \forall \. i  \.\big\}.
$$
Let  \ts $\de_A=\operatorname{Hdim}(\fF_A)$ \ts denote the Hausdorff dimension of $\fF_A$; it is classical that $\de_A\nearrow 1$ as $A\to\infty$.
Then the more precise version of Theorem~\ref{t:BK} is the following.
\begin{thm}[{\rm \cite{BK14}}{}]\label{t:BK'}
There is a constant $\de_0<1$, such that, for any $A$ with $\de_A>\de_0$,  Zaremba's Conjecture~\ref{conj:Zaremba} holds for a positive proportion of $\jj$.
\end{thm}

We explain in \S\ref{s:BK} how the same methods can be adapted to prove an analogous theorem in the context of \S\ref{ss:intro-CF}. For $A>0$, define the limit set:
$$
\fC_A \, := \, \big\{ [b_1,1,b_2,1,\ldots] \ : \ 
1\le b_i \le A, \ \ \forall \. i   \.\big\},
$$
and let
$$
\vt_A := \operatorname{Hdim}(\fC_A)
$$
be its Hausdorff dimension.
\begin{thm}[{\rm Main Diophantine Theorem}{}]\label{thm:BK-dim-strong}
There is a constant $\de_0<1$, such that, for any $A>0$ with $\vt_A>\de_0$,
the Diophantine Conjecture~\ref{conj:Zaremba-strong} holds for a positive proportion of $\kk$.
\end{thm}
Let
\begin{equation}\label{eq:fCinfty}
\fC_\infty=\bigcup_{A=1}^\infty \fC_A,
\end{equation}
and $\vt_\infty:=\operatorname{Hdim}\fC_\infty$, so that $\vt_A\nearrow \vt_\infty$ as $A\to\infty$.
This Cantor set
 is closely related to the fractal
\begin{equation}\label{eq:barfCDef}
\bar\fC : =  \big\{ [b_1,1,b_2,1,\ldots] \ : \  b_i \ge1, \ \ \forall \. i  \.\big\},
\end{equation}
see Figure \ref{fig:Cinf}.
In particular, the former is the intersection of the latter with the set of badly approximable numbers (that is, those with bounded partial quotients).
Hence the Hausdorff dimension
$$
\bar\vt :=\operatorname{Hdim}(\bar\fC)
$$
is an upper bound for $\vt_\infty$.

Curiously, fractal \ts $\bar\fC$ \ts appeared recently in work of Han\v{c}l--Turek \cite{HanclTurek2023}, where they showed, among other things,
an analogue of Hall's famous \emph{sum-set projection theorem} \cite{Hall1947}, namely that every $x\in[0,1]$ can be written as the sum of two elements from $\bar\fC$. They also estimated its Hausdorff dimension, proving that
\begin{equation}\label{eq:barvtEst}
0.732 \ < \ \bar\vt \ < \ 0.819.
\end{equation}

\begin{figure}
\includegraphics[width=2in]{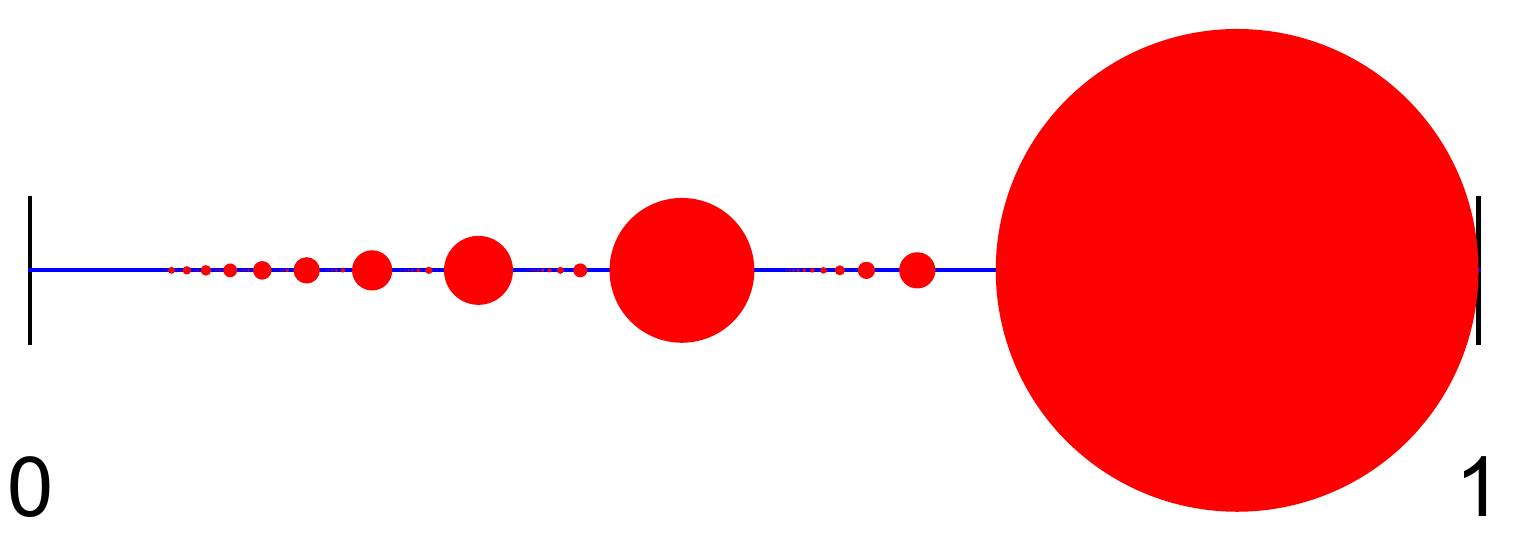}
\caption{The fractal $\bar\fC$ defined in \eqref{eq:barfCDef}. Diameters of circles along the unit interval correspond to ranges being removed at each stage.}
\label{fig:Cinf}
\end{figure}

In order to conclude Theorem~\ref{thm:BK-strong}, we thus need to verify that the dimension $\vt_\infty$ indeed exceeds the numerical value of $\de_0$ coming from the proof of Theorem~\ref{thm:BK-dim-strong}.
The original work \cite{BK14} gave the value of $\de_0=0.984$, and there have been a number of subsequent improvements bringing down this critical parameter.
As explained in Section~\ref{s:BK}, the proof of (a density-one version of) Theorem~\ref{t:BK'} follows the ``orbital circle method'' which relies on an analysis of a ``minor arcs'' estimate and ``major arcs'' estimate.

In \cite{FK14}, Frolenkov--Kan gave an improved treatment of the minor arcs  only, allowing them to reduce $\de_0$ to $5/6=0.8\bar3$, but at the cost of concluding that Zaremba's conjecture holds for only a positive proportion of $\kk$, rather than a density one set.
 In \cite{Hua15}, Huang was able to  combine the major arcs analysis from \cite{BK14} with the minor arcs treatment in \cite{FK14} to recover $\de_0=5/6$ for a density one set. Since then, there have been a number of improvements on the minor arcs analysis by Kan, who proved the positive proportion version with:
\begin{itemize}
\item $\de_0\, =\, 4/5\, =\, 0.8$ in \cite{Kan15},
\item $\de_0\, =\,  \frac14(\sqrt{17}-1)\,  \approx\, 0.781$ in \cite{Kan17}, and most recently
\item $\de_0 \, =\,  \frac13(\sqrt{40}-4)\,  \approx\,  0.775$, see \cite{Kan21}.
\end{itemize}

\smallskip

\begin{thm}[{\rm cf.~\cite{Kan21}}{}]\label{thm:Kan}
Theorem~\ref{thm:BK-dim-strong} holds with $\de_0=0.775$.
\end{thm}

Unfortunately, the estimate \eqref{eq:barvtEst} is completely inconclusive! It neither shows that $\vt_\infty$ is too small to hope for it to exceed $\delta_0$, nor does it guarantee that we are in the range of exceeding $\delta_0$ and hence of being able to apply Theorem~\ref{thm:Kan}.
It thus remains to exhibit rigorous estimates for $\vt_\infty$ that settle the question. 

\begin{rem}\label{rmk:Huang}
Combined with Huang's method, it is possible to prove that density one sets of~$\kk$ satisfy Zaremba's conjecture, with these same parameters of $\de_0$ in Theorem~\ref{thm:BK-dim-strong}.
While doing this properly with all details would multiply the length of this paper five-fold, we discuss the key steps in \S\ref{ss:BK-orbital}; executing this would make Theorems \ref{thm:BK-dim-strong}, \ref{thm:BK-strong}, \ref{thm:main2}, and \ref{t:main} hold with ``positive proportion'' replaced by ``density one,'' but would not affect the Main Theorem~\ref{t:main-exp}.
\end{rem}

\smallskip

\subsection{Computing Hausdorff dimensions}\label{ss:intro-H-dim}
The numerical estimation of Hausdorff dimensions of fractals corresponding to various families of
continued fractions is a major object of study in dynamical systems.  Specifically,
estimating \ts $\de_A$ \ts has become a benchmark problem in the area.

Famously, Good \cite{Good41} showed that \ts $0.5194< \de_2 <0.5433$, greatly
improving upon Jarn\'{\i}k's 1928 lower bound \ts $\de_2 \ge 0.25\ts$.
Bumby \cite{Bum85} used Good's approach and computer assistance
to calculate \ts $0.5312< \de_2 <0.5314$.  Although the exact value is not known,
further algorithmic and computational advances led to better estimates,
see  \cite[\href{https://oeis.org/A279903}{A279903}]{OEIS}.
Notably, Hensley \cite{Hen96} improved the bounds
to $\de_2 \approx 0.53128051$, and eventually this value was computed
to 25 digits \cite{JP01}, 54 digits \cite{Jen04}, 100 digits \cite{JP18},
and most recently to 200 digits \cite{PV22}.

Implementing the Pollicott--Vytnova algorithm  \cite{PV22} in our setting, we are able to greatly improve on \eqref{eq:barvtEst} and show the following.

\begin{thm}[{\rm Main Dimension Theorem}{}]\label{thm:H-dim}
Let \ts $\vt_\infty$ \ts denote the Hausdorff dimension of the fractal
$\fC_\infty$ in \eqref{eq:fCinfty}.
Then \. $ \vt_\infty > \delta_0$, for $\delta_0=0.775$.
In particular, for $A=\Ais$, we already have that $\vt_A > \delta_0 = 0.775$.
\end{thm}

This finally allows us to apply the Main Diophantine Theorem~\ref{thm:BK-dim-strong} and Theorem~\ref{thm:Kan} to conclude  Theorem~\ref{thm:BK-strong}.

\begin{rem}\label{rmk:vtInfty}
We have made no effort to optimize the value of $A=\Ais$.
That said, note that it is possible to prove also the upper bound $\vt_\infty\le \bar\vt< 0.799$ (see \S\ref{sec:vtInfty}), showing that $\vt_\infty$ exceeds \emph{in the  hundredths place} the best available value of the parameter $\de_0=0.775$.
\end{rem}

\smallskip

\subsection{Paper structure} \label{ss:intro-structure}
As explained above,  Theorem~\ref{t:main} follows from a combination of
the Main Graph Theorem~\ref{thm:graphMain}, the Main Diophantine Theorem~\ref{thm:BK-dim-strong}, and the Main Dimension Theorem~\ref{thm:H-dim}.
We present proofs of each in separate sections (Sections~\ref{s:CF}--\ref{s:H-dim}).
	We conclude with final remarks and open problems in Section~\ref{s:finref}.

\medskip


%


\section{Planar graphs and continued fractions}\label{s:CF}

\subsection{Recursion identities}\label{ss:CF-rec}
Let \ts $G=(V,E)$ \ts be a graph without loops, but possibly with
multiple edges.  Graph \ts $G$ \ts is called \emph{simple} if it
does not have multiple edges.

For an edge \ts $e\in E$, denote by \ts $G-e$ \ts the
graph obtained from $G$ by deleting the edge~$e$.
Similarly, denote by \ts  $G/e$ \ts the graph obtained by
contracting vertices incident to $e$ and removing the
resulting loops.
As in the introduction, denote by \ts $\tau(G)$ \ts the number
of spanning trees in~$G$.
Recall that \ts $\tau(G) = \tau(G-e) + \tau(G/e)$, see e.g.\ \cite[Thm~II.18]{Tut84}.

Let \ts $G=(V,E)$ \ts be a  graph and let \ts $e \in E$ \ts be an edge in~$G$.
A pair \ts $(G,e)$ \ts is called a \defn{marked graph}.
Define the \defn{spanning tree vector} \ts of the marked graph as follows:
$$
\bv(G,e) \, := \, (\tau(G-e),\tau(G/e))\..
$$

We consider two dual operations on marked graphs.
Denote by \. $\f^k(G,e)$ \.  the marked graph \ts $(G',e')$ \ts
obtained by replacing edge~$e$ with a path of length $(k+1)$, and marking one of the
new edges as~$e'$ (see Figure~\ref{f:two-oper}).
Note that if \ts $G$ \ts is connected, simple and planar, then
so is~$G'$.
Denote \ts $G'=(E',V')$, and observe that \ts  $|E'| = |E| + k$,
$|V'| = |V| + k$.


\begin{lemma}\label{lem:rec-f}
Suppose that \ts $\at(G-e)>0$ \ts and \ts $\at(G/e)>0$. Then we have that
$$\bv(G',e') \, = \, \bv(G,e)\cdot \mattwo 1k01.$$
\end{lemma}

\begin{proof}
Note that \ts $\f^k = \f^1\big(\f^{k-1}\big)$.  Thus it suffices to prove the lemma for $k=1$,
as the general case follows by induction.  Observe that \. $\at\big(G'-e') =  \at(G-e)$,
since the new edge created by $\f$ is a bridge to an endpoint in \ts $G'-e'$.  We also have \ts $G'/e' \simeq G$ \ts
by the construction of~$\f$, which gives \ts $\at\big(G'/e'\big) = \tau(G) = \at(G-e) + \at(G/e)$.
This gives the result.
\end{proof}


\smallskip

For the second operation,
denote by \. $\g^k(G,e)$ \.  the marked graph \ts $(G'',e'')$ \ts obtained
by replacing edge~$e$ with $(k+1)$ parallel edges, and marking one of the
new edges as~$e''$  (see Figure~\ref{f:two-oper}).
Note that if \ts $G$ \ts is connected and planar, then
so is~$G''$, but simplicity is not in general preserved, since we are creating multiple edges.  Denote \ts $G''=(E'',V'')$, and observe that \ts  $|E''| = |E| + k$,
$|V''| = |V|$.

\begin{lemma}\label{lem:rec-g}
Suppose that \ts $\at(G-e)>0$ \ts
and \ts $\at(G/e)>0$. Then we have:
$$\bv(G'',e'') \, = \, \bv(G,e)\cdot \mattwo 10k1.$$
\end{lemma}

The proof is similar to the proof of Lemma~\ref{lem:rec-f} and will be omitted.
In fact, Lemma~\ref{lem:rec-g} follows from Lemma~\ref{lem:rec-f} via matroid
duality, but we will not need this observation.  Finally, note that
\begin{equation}\label{eq:gcd}
\gcd\big(\at(G'-e'), \at(G'/e')\big) \ = \   \gcd\big(\at(G''-e''), \at(G''/e'')\big)
\ = \   \gcd\big(\at(G-e), \at(G/e)\big),
\end{equation}
for all \ts $k\ge 1$.

\begin{figure}[hbt]
 \begin{center}
   \includegraphics[height=2.95cm]{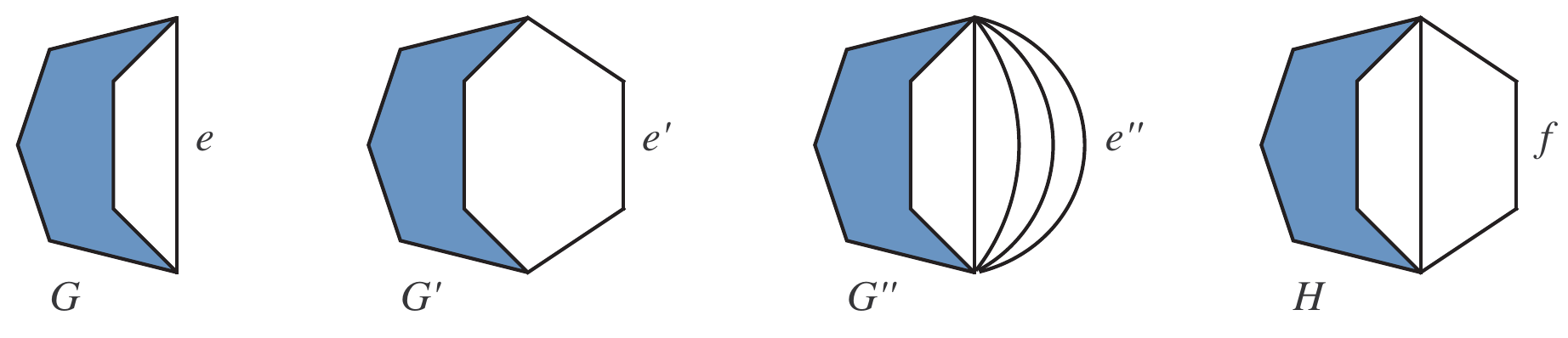}
   \vskip-.2cm
   \caption{Marked graphs $(G,e)$, \. $(G',e')=\f^2(G,e)$, \. $(G'',e'')=\g^3(G,e)$ \.
   and \. $(H,f) = \h^2(G,e)$.}
   \label{f:two-oper}
 \end{center}
\end{figure}

\smallskip

\subsection{Proof of Theorem~\ref{thm:graphMain}} \label{ss:CF-proof}
We will prove the theorem as a consequence of the following slightly more general lemma.

\smallskip

\begin{lemma} \label{l:CF-graph-2}
	Let \ts $\kk,\jj \geq 1$ \ts  be positive integers with \ts $\kk \leq  \jj$ \ts and  \ts $\gcd(\kk,\jj)=1$.
	Suppose that
	\begin{equation}\label{eq:CFEagain}
	\frac{\kk}{\jj} \ = \   [b_1,1,b_2,1,\ldots, b_m,1],
	\end{equation}
	for some \. $b_1,\ldots, b_m\ge 1$.
	Then there exists a simple planar graph \ts $G=(V,E)$ \ts and  edge \ts $e\in E$, such that
	\[
		\tau(G-e) \. = \. \kk, \quad \tau(G/e) \. = \. \jj, \quad \text{and} \quad |V| \. = \.  b_1+\ldots+b_m+2\ts.
	\]
\end{lemma}

\smallskip

\begin{proof}
Let $G=(V,E)$ be a connected simple planar graph with a marked edge \ts $e=(x,y) \in E$.
For \ts $k \geq 1$, consider a marked graph
$$
\h^k(G,e) :=\f^k\big(\g(G,e)\big)
$$
obtained by adding an edge parallel to~$e$, and subdividing~$e$ with $k$ vertices.
This operation is equivalent to adding a path with $k+1$ edges
connecting $x$ and~$y$, and marking
one of the edges in the new path
 (see Figure~\ref{f:two-oper}).

Note that marked graph  \. $(H,f):=\h^k(G,e)$ \. is a connected simple planar
graph by the description above.  Combining Lemma~\ref{lem:rec-f} and Lemma~\ref{lem:rec-g},
we have:
$$
\bv(H,f) \ = \
\bv(G,e)\cdot \mattwo 1011\mattwo1k01
.
$$

Here is a trivial observation (one made already in a completely unrelated setting, namely \cite[p. 1042]{KontorovichMcNamaraWilliamson2017}):
$$
\mattwo
1011
\mattwo
1{k}01
=
\mattwo
0111
\mattwo
011k
.
$$
Thus
\begin{equation}\label{eq:Ups}
\bv(H,f) \ = \
\bv(G,e)\cdot \mattwo 0111\mattwo011k
.
\end{equation}

Now let $P_1$ be a marked graph consisting of a single marked edge $u$ connecting
two vertices.
Then $\bv(P_1,u)=(0,1)$.
Let
 \ts  $(G,e)$ \ts be a marked graph given by
\[  (G,e) \, := \,   \h^{b_1} \h^{b_2} \cdots \h^{b_{m}} (P_1,u),
\]
see Figure~\ref{f:oper-end} (and compare to Figure~\ref{fig:1}).

\begin{figure}[hbt]
 \begin{center}
   \includegraphics[height=3.1cm]{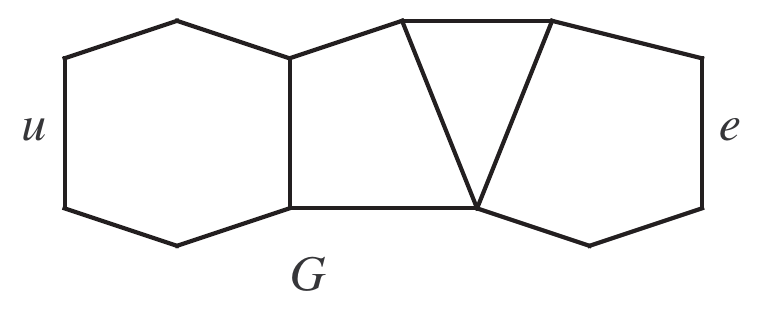}
   \vskip-.5cm
   \caption{Marked graph $(G,e) = \h^3\h^1\h^2\h^4(P_1,u)$.}
   \label{f:oper-end}
 \end{center}
\end{figure}

From above, \ts $G=(V,E)$ \ts is a connected simple planar graph.
By induction, we have:
\[  |V| \, = \,  b_1\. + \. \ldots \. + \. b_m \. + \. 2\ts.
\]
Similarly, equation \eqref{eq:Ups} gives by induction:
\begin{equation}\label{eq:bvGe2}
\bv(G,e) \, = \,
(0,1)\cdot\mattwo 0111\mattwo011{b_m}\cdots\mattwo 0111\mattwo011{b_1}.
\end{equation}
There is a well-known general association between the
matrix product:
\begin{equation}\label{eq:matCFE}
\mattwo
011{a_\ell}
\cdots
\mattwo
011{a_1}
 \, = \,
\mattwo * *\kk\jj
\end{equation}
and
the continued fraction expansion:
$$
\frac \kk\jj \, = \, [0;a_1,a_2,\dots,a_\ell].
$$
It follows that
$$
\bv(G,e)=(\tau(G-e),\tau(G/e))=(\kk,\jj),
$$
where $\gcd(\kk,\jj)=1$ and  $\kk/\jj$ has continued fraction expansion as in \eqref{eq:CFEagain}. 
This completes the proof.
\end{proof}

Finally, we can now prove Theorem~\ref{thm:graphMain}.
From Lemma~\ref{l:CF-graph-2}, we know that the graph $G-e$ (shown in Figure~\ref{fig:2}) has $\tau(G-e)=t$.


\begin{figure}[hbt]
\begin{center}
{\small
\begin{tikzpicture}[
    vertex/.style={circle, fill, inner sep=1.5pt},
    node distance=0.5cm
]
    \node[vertex] (a1) {};
    \node[vertex, below=2cm of a1] (a2) {};
    \draw (a1) -- (a2);

    \node[vertex, right=2cm of a1] (b1) {};
    \node[vertex, below=0.5cm of b1] (b2) {};
    \node[vertex, below=0.5cm of b2] (b3) {};
    \draw[dotted] (b3) -- +(0,-0.5cm);
    \node[vertex, below=.7cm of b3] (b4) {};
    \draw (b1) -- (b2) -- (b3);
    \draw (b3) -- (b4);

    \node[vertex, right=2cm of b1] (c1) {};
    \node[vertex, below=0.5cm of c1] (c2) {};
    \node[vertex, below=0.5cm of c2] (c3) {};
    \draw[dotted] (c3) -- +(0,-0.5cm);
    \node[vertex, below=.7cm of c3] (c4) {};
    \draw (c1) -- (c2) -- (c3);
    \draw (c3) -- (c4);

    \node[vertex, right=2cm of c1] (d1) {};
    \node[vertex, below=0.5cm of d1] (d2) {};
    \node[vertex, below=0.5cm of d2] (d3) {};
    \draw[dotted] (d3) -- +(0,-0.5cm);
    \node[vertex, below=.7cm of d3] (d4) {};
    \draw (d1) -- (d2) -- (d3);
    \draw (d3) -- (d4);

    \node[vertex, right=2cm of d1] (e1) {};
    \node[vertex, below=0.5cm of e1] (e2) {};
    \node[vertex, below=0.5cm of e2] (e3) {};
    \draw[dotted] (e3) -- +(0,-0.5cm);
    \node[vertex, below=.7cm of e3] (e4) {};
    \draw (e2) -- (e3);
    \draw (e3) -- (e4);

    \draw[dashed] ($(b1)-(0,1.1)$) ellipse (0.6cm and 1.4cm);
    \draw[dashed] ($(c1)+(0,-1.1)$) ellipse (0.6cm and 1.4cm);
    \draw[dashed] ($(d1)+(0,-1.1)$) ellipse (0.6cm and 1.4cm);
    \draw[dashed] ($(e1)+(0,-1.1)$) ellipse (0.6cm and 1.4cm);

    \draw[bend left=30] (a1) to (b1);
    \draw[bend left=30] (a1) to (c1);
    \draw[bend left=30] (a1) to (d1);
    \draw[bend left=30] (a1) to (e1);

    \draw (a2) -- (b4);
    \draw (b1) -- (c4);
    \draw (c1) -- ($(c1) !0.33! (d4)$);
    \draw ($(c1) !0.67! (d4)$) -- (d4);
    \draw (d1) -- (e4);

    \node[right=.5cm of c1] {$\cdots$};
    \node at ($(c1)!0.5!(d4)$) {\(\ddots\)};
    \node[above=.1cm of b4] {$\vdots$};
    \node[above=.1cm of c4] {$\vdots$};
    \node[above=.1cm of d4] {$\vdots$};
    \node[above=.1cm of e4] {$\vdots$};

    \node[below=0.3cm of b4] {$b_m$};
    \node[below=0.3cm of c4] {$b_{m-1}$};
    \node[below=0.3cm of d4] {$b_2$};
    \node[below=0.3cm of e4] {$b_1$};
\end{tikzpicture}
}\end{center}
\vskip-.3cm
\caption{A typical graph $G-e$ from Lemma~\ref{l:CF-graph-2}}
\label{fig:2}
\end{figure}

This graph leaves ``tails'' of paths, which can be trimmed without changing the number of spanning trees, resulting in the graph shown in Figure~\ref{fig:1}.

\smallskip


\section{Bourgain--Kontorovich technology}\label{s:BK}

\subsection{Thin orbits}\label{ss:thinOrbits}
Given $A>0$, let
$\cR_A$ denote the set of fractions of interest,
\begin{equation}\label{eq:RA-intro}
\cR_A \,  := \, \big \{ \. \tfrac{\kk}{\jj} = [b_1,1,b_2,1,\ldots,b_m,1] \ : \ 
1\le b_i \le A \ \  \text{for all} \ \ 1\le i \le m\.\big\}
\end{equation}
and let $\cN_A$ be the set of its numerators,
$$
\cN_A \, := \, \big \{ \ts \kk \in \nn \ : \ \tfrac{\kk}{\jj} \in \cR_A \ \  \text{for some} \  \jj > \kk \ts, \ \gcd(\kk,\jj)=1 \ts \big\}.
$$
As already suggested by \eqref{eq:matCFE}, a starting point of the analysis is to convert questions about continued fractions into ones about ``Thin Orbits,'' so that we may use techniques from the
 Orbital Circle Method (see \cite[\S5]{Kon13}). To this end, it is natural to introduce the matrix semigroup
$$
\G_A^{(0)} \, := \, \left\<  \begin{psmallmatrix}
  0 & 1  \\
  1 & a
\end{psmallmatrix} \.
 : \. 1\le a\le A \right\>^+ \ \cap \ \SL(2,\Z)
$$
which corresponds to all even-length words in matrices of the kind appearing in \eqref{eq:matCFE}, and having all ``partial quotients'' $a$ bounded by $A$.

For the application to $\cR_A$ in \eqref{eq:RA-intro}, we introduce the corresponding semigroup
$$
\G_A :=
\left\<
\begin{psmallmatrix}
  0 & 1  \\
  1 & 1
\end{psmallmatrix}
\begin{psmallmatrix}
  0 & 1  \\
  1 & b
\end{psmallmatrix}
\, : \,
1\le b \le A \right\>^+
,
$$
so that $\mattwos **\kk\jj \in \G_A$ if and only if $\kk/\jj \in \cR_A$.
Then to access the numerator $\kk$ in $\cN_A$, we simply take the inner product:
\begin{equation}\label{eq:cNAis}
\cN_A \ = \ \<\bv_1 \cdot \G_A, \bv_2\>,
\end{equation}
where $\bv_1:=(0,1)$ and $\bv_2=(1,0)$.

\begin{rem}\label{rmk:NumDenom}
While the Diophantine Conjecture~\ref{conj:Zaremba-strong} deals with numerators, Zaremba's Conjecture~\ref{conj:Zaremba} is about denominators.  In light of \eqref{eq:matCFE}, this merely amounts to replacing $\bv_2=(1,0)$ above with $\bv_2=(0,1)$.
\end{rem}
\smallskip

\subsection{Proof of Theorem~\ref{t:main-exp}}\label{sec:Pf1p1}
It is already possible to prove Theorem~\ref{thm:main1p1}, which implies Theorem~\ref{t:main-exp}. Indeed, let
 $B_N=B_N^{(A)}$ denote the intersection of $\G_A$ with a ball of radius $N$ in $\SL(2,\rr)\subset \rr^4$ with respect to a fixed archimedean norm, and let
\begin{equation}\label{eq:RNdef}
R_N (n) \, := \, \sum_{\ga\in B_N}\textbf{1}_{\{\<\bv_1\ga,\ts \bv_2\> \ts = \ts n\} }
\end{equation}
be the ``representation number,'' that is, the multiplicity with which an integer $n$ of size roughly $N$ occurs in $\cN_A$ from \eqref{eq:cNAis}.
The sum of these representation numbers is of course just the size of the ball, which is given (see \cite{Lalley1989, Hensley1989}) by
\begin{equation}\label{eq:Lalley}
\sum_{n}R_N(n) \, = \, |B_N| \, = \, N^{2\vt_A+o(1)},
\end{equation}
as $N\to\infty$. It is easy to see that $R_N(n)\ll N$, which implies  the crude lower bound
$$
\cN_A\cap[1,N] \ \gg \  N^{2\vt_A-1-o(1)}.
$$
This already gives Theorem~\ref{thm:main1p1}, as long as $\vt_A>1/2$, which happens already for $A=4$ (see Remark \ref{rmk:A4}).

It is actually possible to do quite a bit better, with relatively little work, due to the sum-set structure hidden within $\cN_A$. Indeed, suppose that $t/u\in\cR_A$, that is, $\mattwos **\kk\jj \in \G_A$. Then $\cR_A$ also contains
$$
\cfrac{1}{1  \. +  \. \cfrac{1}{1  \. +  \.  \frac{t}{u}}}\, = \, \frac{t+u}{t+2 u},
$$
that is, $\cN_A$ must also contain the numerator $t+u$. Of course $t+u$ and $t$ together determine the pair $(t,u)$. Therefore \eqref{eq:Lalley} actually implies the stronger lower bound
$$
\cN_A\cap[1,N] \ \gg \ N^{\vt_A-o(1)}.
$$
This gives Theorem~\ref{thm:main1p1}, for any $A$ having $\vt_A>0$, including $A=2$. For even stronger implications in this direction, see, e.g., \cite[Theorem 1.23]{BK14}, and developments thereafter.
\smallskip

\subsection{Local-global principles and admissibility}
Next we move on to the proof of Theorem~\ref{thm:BK-dim-strong}, en route to Theorem~\ref{t:main}.
The general Bourgain--Kontorovich machinery  proves results of the following kind, of which Theorem~\ref{thm:BK-dim-strong} is a special case.
Let $\G\subseteq \G^{(0)}_A$ be a finitely generated semigroup,
and for fixed $\bv_1,\bv_2\in\Z^2$, consider the set
$$
\cO \, := \, \<\bv_1 \G,\bv_2\>.
$$
We say that an integer $n$ is \emph{admissible}
if
$$
n\in\cO \. (\text{mod}~q)  \quad \text{for all integers} \ \. q.
$$
Let \ts $\cA=\cA(\G,\bv_1,\bv_2)$ \ts denote the set of admissible integers, and note that \ts $\cO \subseteq \cA$. Then the even more precise density-one version of Theorem~\ref{t:BK'} is the following (see \cite[Theorem 1.8]{BK14}).

\smallskip

\begin{thm}
\label{t:BK1}
There is a constant \ts $\de_0 < 1$ \ts such that, if the Hausdorff dimension $\de_\G$ of the  limit set  of $\G$ satisfies
$\de_\G \. > \. \de_0\ts,$
then $\cO$ contains a density one set of its admissible integers:
\begin{equation}\label{eq:O-admissible}
\frac{\big|\cO \cap \{1,\ldots,N\}\big|}{\big|\cA\cap \{1,\ldots,N\}\big|} \,
\to \, 1 \quad \text{as} \quad N\to\infty.
\end{equation}
\end{thm}

This result is interpreted as an asymptotic ``local-global'' principle for the orbit $\cO$.
See \S\ref{ss:finrem-asy} for further discussion on this result and the rate of convergence in \eqref{eq:O-admissible}.
To adapt it to our setting, the first question is to determine the set of admissible integers. In practice, this is quite simple, and in this case follows from the following.
\begin{lemma}\label{lem:GAmodq}
 Let $A\ge2$.  Then \. $\G_A \ts (\text{\em mod}~q) = \SL(2,\Z/q\Z)$, for all \ts $q\ge 2$.
\end{lemma}
\begin{proof}
Since reduction mod $q$ turns $\G_A$ into a finite group (in which every element has finite order), it follows that the reduction
mod~$q$ of the semigroup \ts $\G_A$ \ts is equal to the reduction mod $q$ of the group generated by \ts $\G_A$.
Let $M_b := \begin{psmallmatrix}
0 & 1  \\
  1 & 1
\end{psmallmatrix} \begin{psmallmatrix}
  0 & 1  \\
  1 & b
\end{psmallmatrix}$.
Note that $M_1^{-1} \cdot M_2 =\begin{psmallmatrix}
  1 & 1  \\
  0 & 1
\end{psmallmatrix}$, and $M_1M_2^{-1}M_1=\begin{psmallmatrix}
  1 & 0  \\
  1 & 1
\end{psmallmatrix}$.  Therefore, the group generated by $M_1$ and $M_2$ is all of $\SL(2,\Z)$.
Hence the mod $q$ reductions of $\G_A$ and $\SL(2,\Z)$ are the same,
and equal to \ts $\SL(2,\Z/q\Z)$.
\end{proof}
It follows immediately from \eqref{eq:cNAis} and Lemma \ref{lem:GAmodq} that all integers are admissible for $\cN_A$.
Then the Main Diophantine Theorem~\ref{thm:BK-dim-strong} follows on applying Theorem~\ref{t:BK1} with $\G=\G_A$.
For the reader's convenience, we give a sketch of some of the ingredients going into the proof of Theorem~\ref{t:BK1}, referring the reader to, e.g., \cite{Kon13} and the original references for more details.

\smallskip

\subsection{Orbital circle method}\label{ss:BK-orbital}
%

To show that an integer \ts $n$ \ts of size about \ts  $N$ \ts 
belongs to $\cO$, 
it suffices to show that it has a positive number of representatives, that is, $R_N(n)>0$ for $R_N$ in \eqref{eq:RNdef}.
%
The Fourier transform of $R_N$ is the exponential sum
$$
S_N(\theta) \, := \, \sum_{\ga\in  B_N}\. \exp(\theta \<v_1\ga,v_0\>),
$$
so that
\begin{equation}\label{eq:RNFourierInv}
R_N(n) \, = \, \int_{\rr/\Z} \. S_N(\theta)\. \exp(-n\theta) \. d\theta.
\end{equation}

The ``circle method'' is the idea akin to signal processing that the above sum may sometimes be estimated effectively by decomposing $\theta$ into the so-called major and minor arcs, the former corresponding to the signal, and the latter meant to represent noise. The major arcs are values of $\theta$ that are ``close'' to fractions with ``small'' denominators (for suitable meanings of these words), and minor arcs are simply the complement. This leads to a decomposition (see \cite[(4.25)]{BK14})
$$
R_N(n) \, = \, \cM_N(n) \. + \. \cE_N(n),
$$
where the ``main'' term $\cM(n)$ is the contribution to \eqref{eq:RNFourierInv} coming from $\theta$ in the major arcs, and $\cE_N(n)$ representing the ``error'' term.

One shows (see, e.g., \cite[Theorem~4.28]{BK14}),
that the main term is at least of the expected order:
\begin{equation}\label{eq:cMNgg}
\cM_N(n) \ > \ C \. \fS(n) \. {| B_N|\over N} \quad \text{for some} \quad C>0,
\end{equation}
where $\fS(n)$ is the ``singular series,'' which vanishes if $n$ is not admissible,
and is roughly of order~$1$ elsewhere. If one could prove that for individual $n=\Theta(N)$,
the error term is significantly smaller,
\begin{equation}\label{eq:cENlittleOh}
\cE_N(n) \overset{?}=o\left({| B_N|\over N}\right),
\end{equation}
then the representation number is asymptotically positive, and hence the full local-global principle holds: every sufficiently large admissible number is represented. Unfortunately, as discussed in~$\S$\ref{ss:finrem-asy}, there are settings which are analytically indistinguishable from this one, for which there are arbitrarily large admissible integers with no representatives; and hence \eqref{eq:cENlittleOh} is false as the error term \emph{must} be exactly as large as the main term infinitely often. Instead, what one manages to prove in practice is an $L^2$-averaged version of \eqref{eq:cENlittleOh}, namely, a statement roughly of the form (see \cite[Theorem 7.1]{BK14} and \cite[Theorem 9.1]{Hua15}):
\begin{equation}\label{eq:L2bnd}
\sum_{cN \le n \le N} \. |\cE_N(n)|^2 \, = \, o\left( N \cdot {| B_N|^2\over N^2}\right).
\end{equation}
An elementary Cauchy--Schwartz type argument concludes Theorem~\ref{t:BK1} from \eqref{eq:cMNgg} and \eqref{eq:L2bnd}.

\smallskip

The analysis of both $\cM_N$ and $\cE_N$ is long and difficult. The former relies on ``expander graph'' properties for certain families Cayley graphs in congruence towers, and the thermodynamic formalism of Ruelle transfer operators. The analysis of the error term $\cE_N$ uses more elementary tools, but involves a rather complicated process coming from the theory of cancellation in exponential sums. It only succeeds to prove \eqref{eq:L2bnd} if the Hausdorff dimension \ts $\de_\Ga$ \ts of the limit set of $\G$ exceeds the threshold value~$\de_0\ts$.

\smallskip

If all one wants is a ``positive proportion'' result, rather than ``density one,'' then one need not introduce major arcs and can avoid proving \eqref{eq:cMNgg}; see~$\S$8 in the original {\tt arXiv} version
of \cite{BK14}.\footnote{\href{https://arxiv.org/abs/1107.3776v1}{arxiv.org/abs/1107.3776v1}}

\smallskip

\subsection{Improvements on $\delta_0$}

As mentioned in~$\S$\ref{ss:intro-Zaremba}, the error term treatment in Frolenkov--Kan \cite{FK14} and series of papers by Kan \cite{Kan15}--\cite{Kan21} improves on the original minor arcs analysis in \cite{BK14}, leading to improvements in $\de_0$.
We give here a few comments on at least some of their ideas, as well as the  modifications that would be needed to improve their positive proportion results to density one.

In order to establish the error bound \eqref{eq:L2bnd}, it is necessary to modify the representation number $R_N$ defined in \eqref{eq:RNdef} to create bilinear (in fact, multilinear) forms, as follows.
As discussed in \cite[\S3.4]{BK14}, one replaces the full norm ball $B_N$ in $\G$ by a suitable ``ensemble,'' $\Omega_N\subset B_N$ which, in the original argument, has something like the  structure
$$
\Omega_N \ \approx \ \widetilde B_{N^{1/2}} \cdot \widetilde B_{N^{1/4}} \cdot \widetilde B_{N^{1/8}}\cdots,
$$
where $\widetilde B_N$ are some large ``modified balls,'' $\widetilde B_N \subset B_N$. We can view this ensemble diagrammatically as follows:
\begin{center}
\begin{tikzpicture}
\draw[thick, fill=blue!20] (0,0) rectangle (1,1);

\draw[thick, fill=blue!20] (1,0) rectangle (1.5,0.5);

\draw[thick, fill=blue!20] (1.5,0) rectangle (1.75,0.25);

\draw[thick, fill=blue!20] (1.75,0) rectangle (1.875,0.125);

\draw[thick, fill=blue!20] (1.875,0) rectangle (1.9375,0.0625);

\draw[thick] (2.3,0.03125) node {$\cdots$};

\end{tikzpicture}
\end{center}
 Using this ensemble as is allows one to prove Zaremba's conjecture for a positive proportion of denominators, but the
modified balls are not
``spectral,'' that is, not
suitable for the major arcs analysis needed for establishing \eqref{eq:cMNgg}. Therefore one adds a tiny spectral sliver to the first ball (for the precise statement, which is rather technical, see \cite[(3.37)]{BK14}),
which can be expressed diagrammatically as:
\begin{center}
\begin{tikzpicture}
\draw[thick, fill=red!20] (0,0) rectangle (.1,1);

\draw[thick, fill=blue!20] (.1,0) rectangle (1,1);

\draw[thick, fill=blue!20] (1,0) rectangle (1.5,0.5);

\draw[thick, fill=blue!20] (1.5,0) rectangle (1.75,0.25);

\draw[thick, fill=blue!20] (1.75,0) rectangle (1.875,0.125);

\draw[thick, fill=blue!20] (1.875,0) rectangle (1.9375,0.0625);

\draw[thick] (2.3,0.03125) node {$\cdots$};

\end{tikzpicture}
\end{center}
This makes the modified ensemble amenable to analysis.

The first innovation in \cite{FK14} is to refine the ensemble $\Omega_N$ even further, to one of the following rough diagrammatic shape (see \cite[(3.5)]{FK14} and \cite[\S6]{Kan21}):

\begin{center}
\begin{tikzpicture}
\draw[thick] (-0.3,0.03125) node {$\cdots$};

\draw[thick, fill=blue!20] (0,0) rectangle (0.0625,0.0625);

\draw[thick, fill=blue!20] (0.0625,0) rectangle (0.1875,0.125);

\draw[thick, fill=blue!20] (0.1875,0) rectangle (0.4375,0.25);

\draw[thick, fill=blue!20] (0.4375,0) rectangle (0.9375,0.5);

\draw[thick, fill=blue!20] (0.9375,0) rectangle (1.9375,1);

\draw[thick, fill=blue!20] (1.9375,0) rectangle (2.4375,0.5);

\draw[thick, fill=blue!20] (2.4375,0) rectangle (2.6875,0.25);

\draw[thick, fill=blue!20] (2.6875,0) rectangle (2.8125,0.125);

\draw[thick, fill=blue!20] (2.8125,0) rectangle (2.875,0.0625);

\draw[thick] (3.3,0.03125) node {$\cdots$};

\end{tikzpicture}
\end{center}
This ensemble is again missing a spectral component, but it and its further refinements lead to the aforementioned improvements on $\delta_0$, and eventually to Theorem~\ref{thm:Kan}. The idea in Huang \cite{Hua15} combines these improvements with the major arcs analysis of \cite{BK14}, by
reinserting another spectral sliver in the largest component (see \cite[(3.51)]{Hua15}), leading to an ensemble of the shape
\begin{center}
\begin{tikzpicture}
\draw[thick] (-0.3,0.03125) node {$\cdots$};

\draw[thick, fill=blue!20] (0,0) rectangle (0.0625,0.0625);

\draw[thick, fill=blue!20] (0.0625,0) rectangle (0.1875,0.125);

\draw[thick, fill=blue!20] (0.1875,0) rectangle (0.4375,0.25);

\draw[thick, fill=blue!20] (0.4375,0) rectangle (0.9375,0.5);

\draw[thick, fill=red!20] (0.9375,0) rectangle (1.0375,1);
\draw[thick, fill=blue!20] (1.0375,0) rectangle (1.9375,1);

\draw[thick, fill=blue!20] (1.9375,0) rectangle (2.4375,0.5);

\draw[thick, fill=blue!20] (2.4375,0) rectangle (2.6875,0.25);

\draw[thick, fill=blue!20] (2.6875,0) rectangle (2.8125,0.125);

\draw[thick, fill=blue!20] (2.8125,0) rectangle (2.875,0.0625);

\draw[thick] (3.3,0.03125) node {$\cdots$};

\end{tikzpicture}
\end{center}
This  again allows one to apply the major arcs analysis (see  \cite[Theorem 5.5]{Hua15}). In combining the minor arcs analysis in \cite{Kan21}
with the major arcs from \cite{BK14, Hua15}, one
 can show
 the density one result in Theorem~\ref{t:BK1}, with the improved values of the parameter $\delta_0$.


\medskip

\section{Estimating the Hausdorff dimension}\label{s:H-dim}

\subsection{Proof of the Main Dimension Theorem~\ref{thm:H-dim}}\label{sec:pfThmHdim}

A well-known strategy to compute Hausdorff dimensions $\vt_A$  comes from techniques in thermodynamics, namely to consider certain Ruelle-type transfer operators, as follows.
For integers $b\ge1$, let $T_b:[0,1]\to[0,1]$ be the operation corresponding to one step in \eqref{eq:bvGe2}, that is, translation by $b$, inversion, translation by $1$ and inversion again,
$$
T_b(x):=\cfrac{1}{1+
\cfrac{1}{b+x}
}
={b+x\over 1+b+x}.
$$
This has derivative
$$
T_b'(x)={1\over (1+b+x)^2}.
$$
 For $A>1$ (we will take $A=\Ais$) and $s\in(0,1)$ (a parameter
used to estimate
 the Hausdorff dimension),
we define the  \defnb{pressure function}, $P=P_A$, given by
$$
P(s) \ := \
\lim_{n\to\infty} \.
\frac1n \.
\log\Bigg(
\sum_{1\le b_1,\dots,b_n\le A}
\big|
\big(T_{b_1}\circ\cdots\circ T_{b_n}\big)'(0)
\big|^s
\Bigg).
$$
It follows from the work (in much greater generality) of Ruelle \cite{Ruelle1982}
(see also Bowen \cite{Bowen1979}), that the pressure function $P$ is a monotone
decreasing function of $s$, and has a unique zero at the Hausdorff dimension \ts $s=\vt_A$.

It is difficult in practice to rigorously approximate the dimension from $P$ alone, and one introduces the following  \defnb{transfer operator} $\sL_{s}=\sL_{s,A}$;
given a function $f:[0,1]\to \R$,
we define:
\begin{equation}\label{eq:LsDef}
[\sL_{s}f](x) \ := \
\sum_{b=1}^A \.
\big|T_b'(x)\big|^s
(f\circ T_b)(x).
\end{equation}
A suitable space on which the transfer operator $\sL_{s}$ can act is, e.g., the
space $\mathscr S$ of real analytic functions on $[0,1]$, with the supremum norm.

The connection to the pressure function and Hausdorff dimension is given by the Ruelle--Perron--Frobenius theorem, which states that $\sL_{s}$ acting on $\mathscr S$
has a unique maximal eigenvalue $\lambda_s = e^{P(s)}$, and the corresponding eigenfunction
\begin{equation}\label{eq:hsDef}
h_s=h_{s,A}\in \mathscr S
\end{equation}
 is strictly positive on $[0,1]$.
 A nice discussion of this theory  is given by Pollicott--Vytnova in \cite[\S2]{PV22},
 where the authors develop a relatively simple and  easily implementable procedure
 to give rigorous estimates for Hausdorff dimensions in a wide variety of settings.

Of particular importance to us is the mechanism for producing lower bounds for dimensions, using certain min-max inequalities.

\begin{lemma}\label{lem:keyPerron}
Suppose that there exists some $s>0$ and $f$ a polynomial which is positive on $[0,1]$, so that
\begin{equation}\label{eq:keyPerron}
[\sL_{s}f] (x) >  f(x),
\end{equation}
 for all $x\in[0,1]$. Then $\vt_A>s$.
\end{lemma}
Intuitively, if we want values of $s$ with $P$ near zero, that corresponds to eigenvalues $\lambda_s=e^{P(s)}$ of $\sL_s$ near $1$. Since $\lambda_s$ is the largest eigenvalue, \eqref{eq:keyPerron} is evidence that $\lambda_s>1$.
\begin{proof}
Since $[0,1]$ is compact, \eqref{eq:keyPerron} implies that there exists some $a>1$ so that $a f(x) \le [\sL_{s}f](x)$ holds for all $x\in[0,1]$.
By \cite[Lemma 3.1(1)]{PV22}, this implies that $a\le e^{P(s)}$, or $P(s)\ge \log a > 0$. But since $P$ is monotonically decreasing and $P(\vt_A)=0$, we have that $s<\vt_A$.
\end{proof}

This reduces the problem to one of producing, for $s=\delta_0=0.775$, say, a sufficiently large value of $A$ and a positive ``test function'' $f$ for which \eqref{eq:keyPerron} holds.
We are grateful to Polina Vytnova for explaining to us the procedure from \cite{PV22} for finding such functions in practice, namely to simply try to approximate the (positive) eigenfunction $h_s$ in \eqref{eq:hsDef} via
Lagrange-Chebyshev interpolation. (In principle, almost any interpolation scheme should work, but this one seems to perform particularly well in this application.)

To this end, fix an order of approximation, $N$ (we will take $N=\Nis$), and for $1\le j \le N$, define the Chebyshev nodes
$$
y_j \ := \ \frac12\left(\cos\left(\frac{2j-1}{2N}\.\pi\right) + 1\right) \ \in \ [0,1],
$$
and corresponding Lagrange interpolation polynomials
$$
\ell_j(x) \ := \ \prod_{1\le k\le N\atop k\ne j}\frac{x-y_k}{ y_j-y_k}.
$$
These are linearly independent, and orthogonal with respect to the counting measure on the nodes; in particular, $\ell_j(y_k)=1$ if $j=k$ and $0$ otherwise.
For $N$ large, the nodes become dense in $[0,1]$, and the Lagrange polynomials span some ``significant'' subspace of our function space $\mathscr S$.

To find a linear combination of the $\ell_j$'s which approximates the eigenfunction $h_s$ of the transfer operator $\sL_{s}$, simply hit them all with $\sL_s$, and evaluate at the nodes. This should excite a harmonic near $h_s$, giving the approximation. In practice, this is achieved as follows. We evaluate the $N\times N$ matrix, $M$, say, having entries
$$
M_{j,k} := [\sL_{s}\ell_j](y_k),
$$
($1\le j,k\le N$),
and
compute a (left) eigenvector $v_s=(v_s^{(1)},\dots,v_s^{(N)})$ corresponding to its largest eigenvalue. By the orthogonality relations, this $v_s$ suggests which linear combination of the $\ell_j$ should have a high correlation with $h_s$; that is, we should try the test function:
$$
f_s:= \sum_{1\le k\le N}v_s^{(j)}\ell_j.
$$
If it turns out that $f_s$ is positive on $[0,1]$, and that \eqref{eq:keyPerron} is satisfied, then we learn from Lemma \ref{lem:keyPerron} that $s$ is a lower bound for $\vt_A$. If not, we should increase $N$ (or if that continues to fail, increase $A$), and hope to get lucky. In  \cite[Prop. 3.10]{PV22}, Pollicott and Vytnova  prove that, if in fact $\vt_A>s$, then such an $f_s$ will eventually be found by taking $N$ large enough.

\begin{figure}[hbt]
    \centering
    \begin{minipage}{0.45\textwidth}
        \centering
        \includegraphics[width=0.9\textwidth]{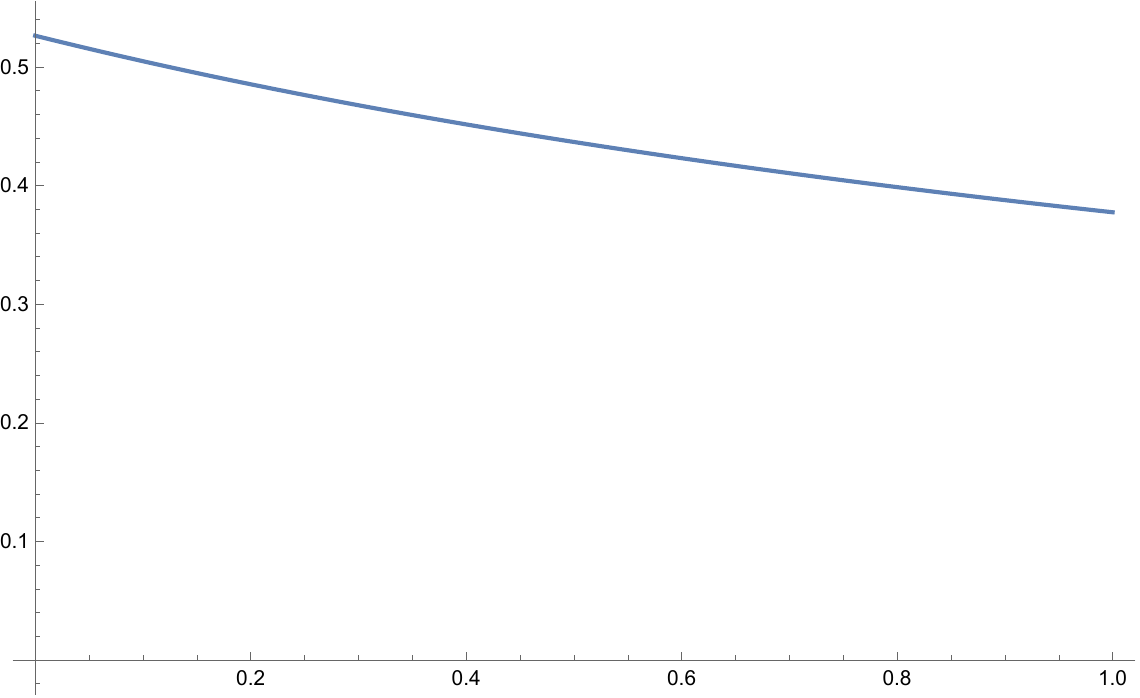} 

        $(a)$
    \end{minipage}\hfill
    \begin{minipage}{0.45\textwidth}
        \centering
        \includegraphics[width=0.9\textwidth]{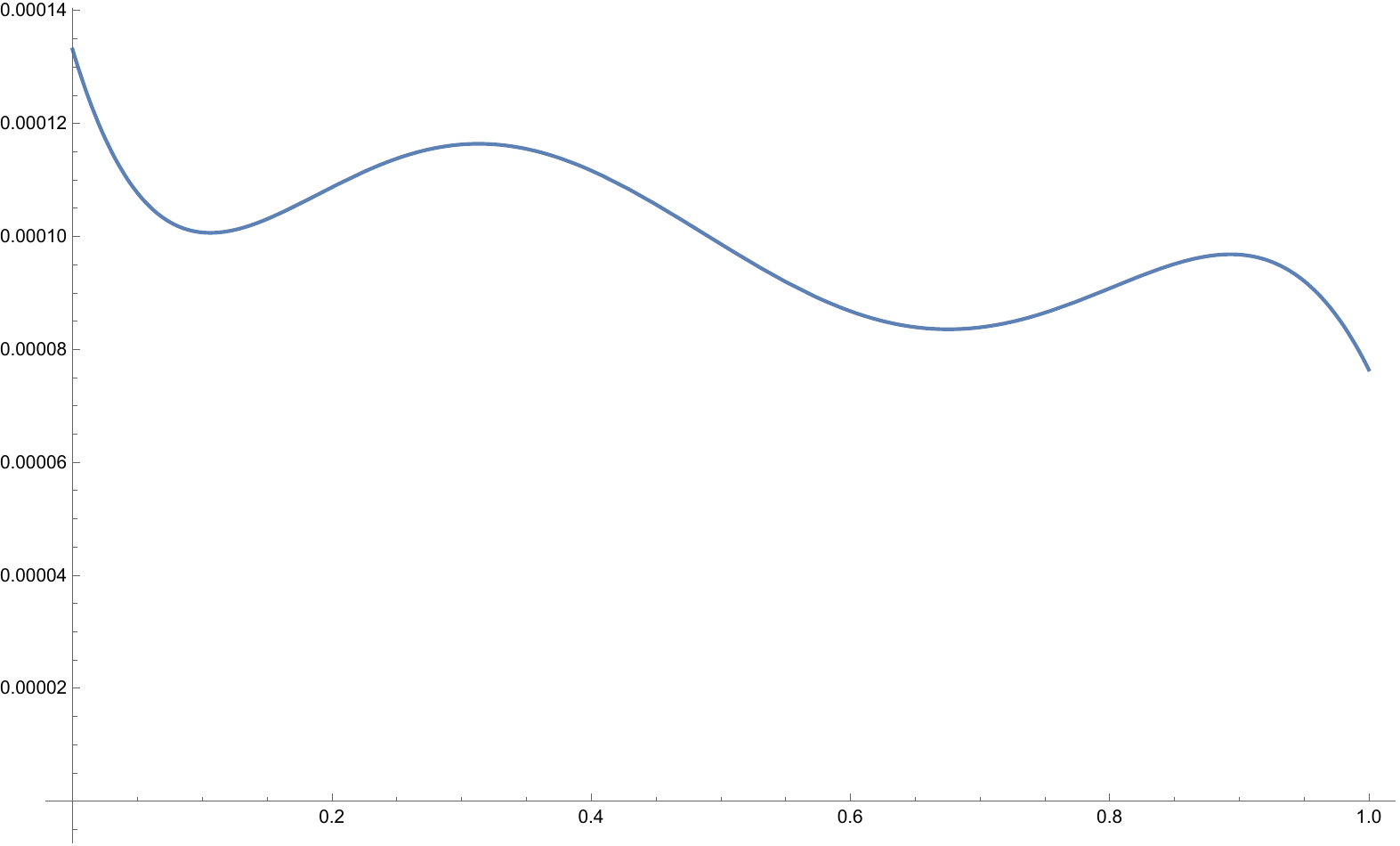} 

        $(b)$
    \end{minipage}
\caption{$(a)$ The test function $f_s$ in \eqref{eq:fsDef}, and $(b)$ difference $\sL_{s}f_s \ - \ f_s$ on $[0,1]$.}
    \label{fig:fs}
\end{figure}

Setting $s=\delta_0=0.775$, $N=\Nis$, and $A=\Ais$, (chosen for the simple reason that $A=100$ fails, even with large values of $N$), the procedure above gives the eigenvector
$$
v_s := (0.3798483, 0.3992862, 0.4366593, 0.4841648, 0.5207676),
$$
and corresponding test polynomial $f_s$ given by:
\begin{equation}\label{eq:fsDef}
f_s(x) := 0.0121844 x^4-0.0513245 x^3+0.116313 x^2-0.225988 x+0.526229.
\end{equation}
For $x\in[0,1]$, it can be verified that $f_s(x)>0.3$, see Figure \ref{fig:fs}$(a)$.
Moreover, it can be verified that on $[0,1]$, the difference $\sL_{s}f_s - f_s$ exceeds $7\times10^{-5}>0$, see Figure \ref{fig:fs}$(b)$. This confirms \eqref{eq:keyPerron}, and hence proves the  Main Dimension Theorem~\ref{thm:H-dim}.

\begin{rem}\label{rmk:A4}
The same techniques easily show that $\vt_A>1/2$ for $A=4$ and higher (and that $\vt_A<1/2$ for $A=2,3$). It is even easier to see that $\vt_A>0$ for $A=2$, as needed in \S\ref{sec:Pf1p1}.
\end{rem}

\smallskip

\subsection{Maximal dimension $\vt_\infty$}\label{sec:vtInfty}

We conclude this section by explaining Remark~\ref{rmk:vtInfty}.
We estimate $\vt_\infty\le \bar \vt$, where the latter is the Hausdorff dimension of the fractal $\bar\fC$ defined in \eqref{eq:barfCDef}.
The  method described in \S\ref{sec:pfThmHdim} cannot be applied directly to the problem of estimating $\bar\vt$,
in particular because the corresponding transfer operator \eqref{eq:LsDef} would now be an infinite sum. Instead, here is an elegant mechanism (again, generously explained to us by Polina Vytnova)
for establishing the estimate
\begin{equation}\label{eq:vtBnd}
\vt_\infty \le \bar \vt < 0.799.
\end{equation}

Let \. $f(x)=\sum_{n=0}^Na_n(x-1)^n$  \.
be a polynomial (expanded about $x=1$).
Then the transfer operator acting on $f$ can be expressed as
\begin{align*}
[\sL_s f](x) \ &: \ =
\sum_{b=1}^\infty \.
|T_b'(x)|^s
\cdot f( T_b(x)) \
= \
\sum_{b=1}^\infty \.
{1\over (1+b+x)^{2s}}
\cdot f\left( 1-{1\over 1+b+x}\right)
\\
&
\ = \
\sum_{b=1}^\infty \.
{1\over (1+b+x)^{2s}} \,
\sum_{n=0}^N \.
a_n\left({-1\over 1+b+x}\right)^n \
= \
\sum_{n=0}^N
a_n
(-1)^n
\zeta(2s+n,2+x)\.
,
\end{align*}
which is now again a finite sum.  Here
$$
\zeta(s,x) \ := \
\sum_{b=0}^\infty \.
{1\over (b+x)^{s}}
$$
is the Hurwitz zeta function.
An analogous statement to Lemma \ref{lem:keyPerron} holds with inequalities reversed: if there is a polynomial $f$ which is positive on $[0,1]$ such that
$$
[\sL_s f](x) < f(x)
$$
for all $x\in[0,1]$, then $\bar\vt<s$.

Now we proceed as before. Set $s=0.799$, and $N=5$, and using the same Chebyshev--Lagrange polynomials, compute the matrix $M_{j,k}=[\sL_s (\ell_j)](y_k)$. (One needs to first extract the polynomials' coefficients.) Its largest eigenvalue has eigenvector
$$
v_s=({0.3820795, 0.4007878, 0.4369026, 0.4830608, 0.5187994}),
$$
\begin{figure}[hbt]
    \centering
    \begin{minipage}{0.45\textwidth}
        \centering
        \includegraphics[width=0.9\textwidth]{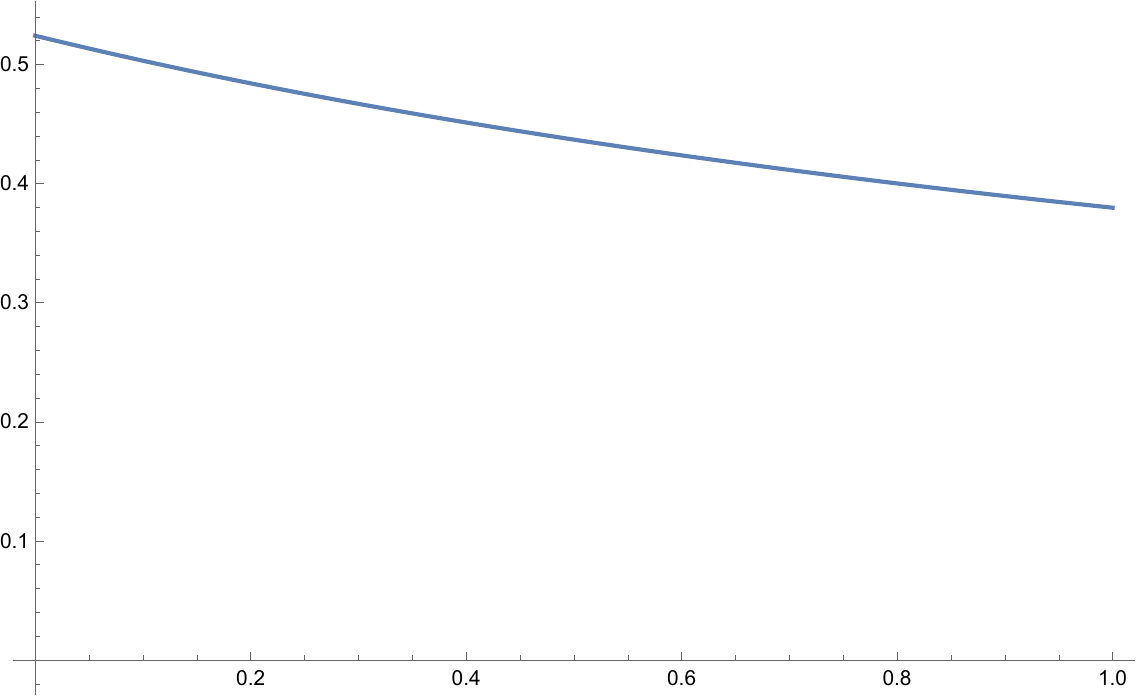} 

        $(a)$
    \end{minipage}\hfill
    \begin{minipage}{0.45\textwidth}
        \centering
        \includegraphics[width=0.9\textwidth]{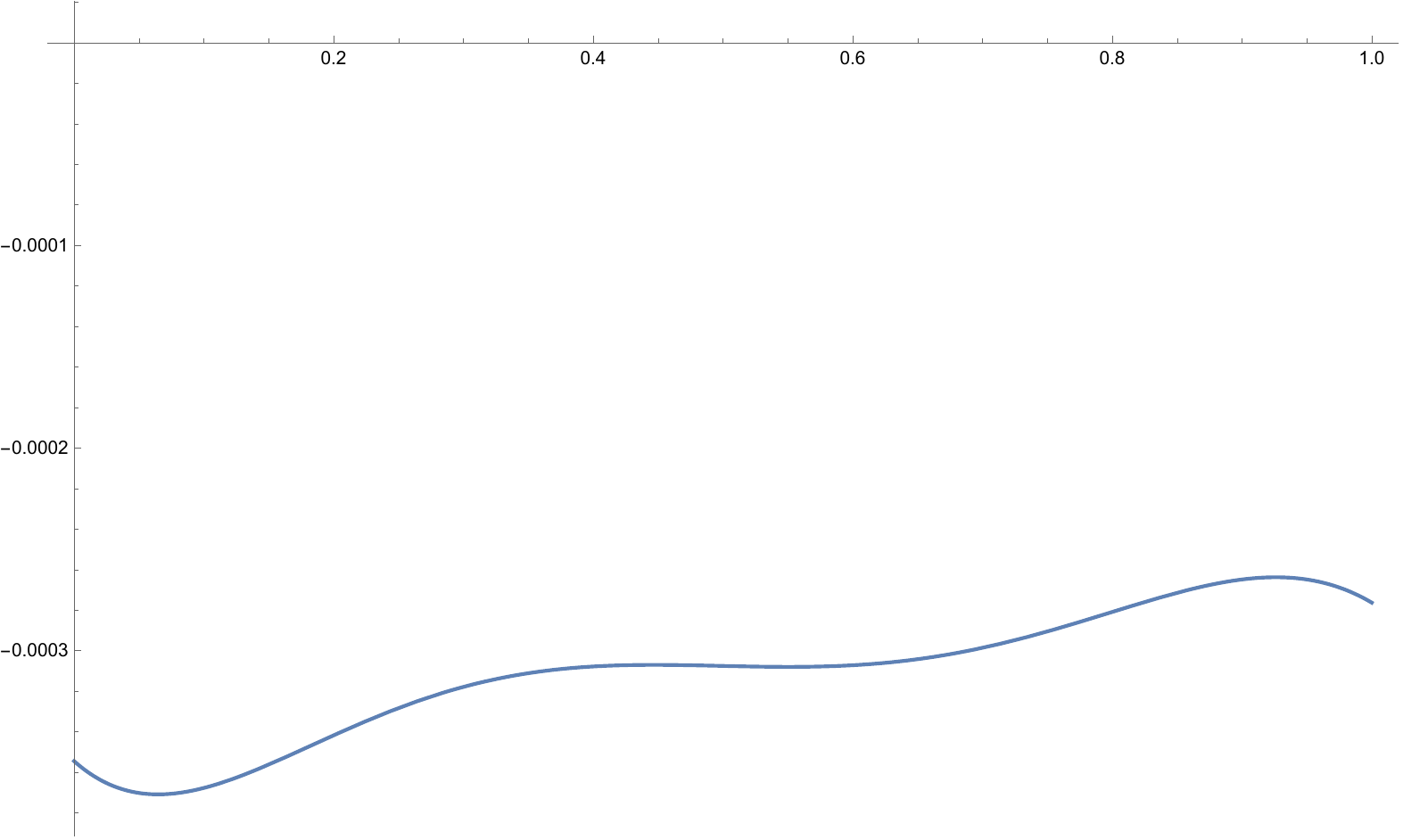} 

        $(b)$
    \end{minipage}
\caption{$(a)$ The test function $f_s$ in \eqref{eq:fs2}, and $(b)$ the difference $\sL_sf_s\, - \, f_s$.}
    \label{fig:Cinfty}
\end{figure}
and the corresponding test function $f_s=\sum v_s^{(j)}\ell_j$ is computed to be
\begin{equation}\label{eq:fs2}
f_s(x)=0.0123381 x^4-0.0517567 x^3+0.116202 x^2-0.221186 x+0.524143.
\end{equation}
This function exceeds $0.3$ on $[0,1]$, and $\sL_s f_s - f_s$ is less than $-0.0002$ on $[0,1]$, see Figure \ref{fig:Cinfty}. This confirms \eqref{eq:vtBnd}.

\medskip

\section{Final remarks and open problems}\label{s:finref}


\subsection{}\label{ss:exp-upper-bnd}
In notation of Main Theorem~\ref{t:main-exp}, one can ask for the
upper bound on the number \ts $|\cT(n)|$ \ts of distinct values of \ts
$\tau(G)$ \ts over all simple planar graphs with $n$ vertices.  The bound \ts $|\cT(n)| < 8^n$
is straightforward; indeed, note that for every planar
simple graph \ts $G=(V,E)$ on~$n$ vertices, we have \ts $\tau(G) < 2^{|E|} < 8^n$,
since \ts $|E|\le 3n-6$.
Writing \ts $\tau(G)$ \ts as a product of eigenvalues of the
Laplacian matrix and using the AM--GM inequality shows that \ts $\tau(G) < 6^n$,
see \cite[Eq.~(1)]{Gri76}.

Skipping over a series
of further improvements, the best known upper bound for the number of spanning
trees in a simple planar graph is \ts $\tau(G)=O(5.2852^n)$ \ts
given in \cite{BS10}.  The optimal base of exponent is probably much smaller,
see below.  However, since
$$
\lim_{n\to \infty} \. \bigl(\max \ts \cT(n)\bigr)^{\frac1n} \, \ge \, 5.0295\ts
$$
(see \cite[$\S$5.1]{Ribo}), this approach has very little room for further improvement.

\subsection{}\label{ss:finrem-limit}
For the lower bound on \ts $|\cT(n)|$ \ts given in Theorem~\ref{t:main-exp}, one can take \ts
$c=1.1103$.  This is derived from the lower bound \ts $\phi^{2\vt_A/(A+1)} \approx 1.1103$ \ts
deduced from 
the proof of Theorem~\ref{thm:main1p1}
in~$\S$\ref{sec:Pf1p1}, with \. $A=3$ \. and \. $0.435<\vt_A<0.436$.
Here \ts $\phi = \tfrac{1+\sqrt{5}}{2} \approx 1.618$ \ts is the
\emph{golden ratio}.  We made no effort in optimizing this constant.
We note, however, that  any lower bound  achievable by our
approach cannot exceed~$\ts\phi$.
Most notably,
in a follow-up to this paper, Alon--Bucic--Gishboliner~\cite{ABG} gave an improved lower bound of $c=1.49$ by  augmenting our approach with a clever counting argument. 
%
%
%
%
%
In view of Theorem~\ref{t:main-exp}, we make the following natural

\begin{conj}\label{conj:set-limit} \. There is a limit \.
$\ga \ts := \ts \lim_{n \to \infty} \ts  \big(|\cT_n|\big)^{\frac{1}{n}}$.
\end{conj}

The conjecture implies that \. $1.49\le \ga \leq 5.2852$,
by the lower and upper bounds given above.
%
%

\subsection{}\label{ss:finrem-graphs}
Denote by  \ts $\cT'(n)$ \ts the set of spanning tree numbers \ts
$\tau(G)$ \ts as $G$ ranges over \emph{all} simple graphs on $n$ vertices,
%
i.e.\ without the planarity assumptions \cite{Sed69,Sed70}.
The sequence \ts $\{|\cT'(n)|\}$ \ts starts as
$$
1, \ 1, \ 2, \ 5, \ 16,  \ 65, \ 386, \ 3700, \ 55784, \ 1134526, \ 27053464\ts,
$$
see \cite[\href{https://oeis.org/A182290}{A182290}]{OEIS}.  By
Theorem~\ref{t:main-exp} and Cayley's formula, we have:
$$
c^n \, \le \, |\cT(n)| \, \le \, |\cT'(n)| \, \le \, \tau(K_n) \, = \, n^{n-2}.
$$
Note that the expected number of spanning trees in a uniform random labeled
graph on $n$ vertices is rather large, and equal to \ts $n^{n-2}/2^{n-1}$.
Note also that the number of unlabeled graphs on $n$ vertices is even larger,
and asymptotically equal to \ts $2^{\binom{n}{2}}/n!\ts$, see e.g.\ \cite[$\S$6.9.2]{Noy15}
and \cite[\href{https://oeis.org/A000088}{A000088}]{OEIS}.
These suggest the following:\footnote{In fact, the first few values
of the sequence suggest a stronger bound: \. $|\cT'(n)| = e^{(1-o(1)) \ts n\log n}$\.. }

\begin{conj}\label{conj:set-all} \. $|\cT'(n)| = e^{\Omega(n\log n)}$.
\end{conj}

Following Sedl\'a\v{c}ek \cite{Sed70}, let
\ts $\al'(\kk)$ \ts denote the minimal number
of vertices of a simple (and not necessarily planar) graph with exactly~$\kk$ spanning trees.
Clearly, we have \ts $\al'(\kk)\le \al(\kk)$, but potentially \ts $\al'(\kk)$ \ts
is much smaller.  Cayley's formula gives a natural
lower bound \ts $\al'(\kk) = \Omega(\log \kk/\log \log \kk)$ \ts in this case.
Sedl\'a\v{c}ek originally conjectured that \ts $\al'(\kk)=o(\log \kk)$ \cite{Sed70}.
Stong's Theorem~\ref{t:Stong} remains the best known upper bound for \ts $\al'(\kk)$.

\subsection{}\label{ss:finrem-det}
By the {matrix-tree theorem},  the number of spanning trees of a simple
graph is a determinant of an integral matrix with off-diagonal entries in~$\{0,-1\}$.
The set of possible determinant values for various classes of combinatorial
matrices is interesting in its own right and closely related to
\emph{Hadamard's maximal determinant problem} and the \emph{determinant
spectrum problem} which remain unresolved.  Given the state of art in these
two problems, we speculate that Conjecture~\ref{conj:main-beta} is likely to be
more accessible than Conjecture~\ref{conj:set-all}.  We refer to \cite{Shah}
and a discussion in \cite[$\S$6.10]{CP-coinc} for further references.
Note also that Conjecture~\ref{conj:set-all} is also related to
the distribution of determinant of random $\{\pm 1\}$ matrix, see
\cite[Conj.~6.8]{Vu21}.

\subsection{}\label{sec:Spec}
The following connection between counting spanning tree numbers and
spectra of operators on locally homogeneous spaces was pointed out
by Peter Sarnak in his Chern lectures \cite{SarnakChern}.
In light of the exponential growth of $\tau$, it is natural to introduce, for a finite graph $G$, the quantity
$$
s(G) \ := \ {\log \tau(G)\over |G|}.
$$
Here, by $|G|$, we mean the number of vertices, $|V|$. Again, by the matrix-tree theorem, this is a ``spectral'' quantity, and it is interesting to investigate, for a family $\cF$ of graphs, the \emph{spanning tree spectrum}:
$$
\text{{\sc Spec}}(\cF) \ := \ \{s(G) \, : \, G\in\cF\}' ,
$$
that is, the set of limit points of $s(G)$, as $G$ ranges in $\cF$.
The density-one version of Theorem~\ref{t:main} then shows, for the family $\cF=\text{{\sc Planar}}$ of simple, planar graphs, that
$$
[0,c) \ \subseteq \ \text{{\sc Spec}}(\text{{\sc Planar}}) \ \subseteq\  [0,C),
$$
for some $0<c<C$, where $c$ comes from Theorem~\ref{t:main} and the upper bound $C$ is related to the discussion
in~$\S$\ref{ss:exp-upper-bnd}.

By contrast, for the family of $k$-regular graphs, $\cF=\text{$k$-\sc{Regular}}$, we have that
$$
\text{{\sc Spec}}(\text{$k$-\sc{Regular}}) \ \subseteq\  (c,C].
$$
for some $0<c<C$. Here the upper bound $C=C_k$ is determined explicitly by work of McKay \cite{McKay1983}, and
the lower bound $c=c_k$ is given by Alon in \cite{Alon1990}.

In the family $\cF=\text{{\sc Simple}}$ of all finite simple graphs, the spanning tree spectrum contains $0$ (due already to trees), and $\infty$ (e.g., from Cayley's formula). It is interesting to investigate these quantities further.

\subsection{}\label{ss:finrem-ave}
In \cite{CP-CF,CP-SY}, the authors used the average cases analysis of the
sum of partial quotients of continued fractions, applying sharp bounds of Larcher
\cite{Lar86} and Rukavishnikova \cite{Ruk11} to get bounds on $\be(\kk)$ discussed
in the introduction.  While these bounds hold for all~$\kk$, the resulting extra
\ts $O(\log \log \kk)$ \ts factors are unavoidable with these tools.
In the terminology of \cite{YK75} (see also  \cite[$\S$4.5.3]{Knuth98}), the sum
of partial quotients is the number of steps of the \emph{subtraction algorithm},
the original (classical) version of the Euclidean algorithm for finding the greatest
common divisor that uses only subtractions instead of divisions.   Finally, in \cite{CP-SY},
it was shown that Zaremba's Conjecture~\ref{conj:Zaremba} implies that \ts $\be(\kk) = \Theta(\log \kk)$.
The authors also produced a series of weaker conjectures, all of which remain open.

\subsection{}\label{ss:finrem-asy}
There is a further refinement in the asymptotic formula in \eqref{eq:O-admissible}.
Namely,  the work of \cite{Bou18} and \cite{MOW19} shows that the asymptotic
value~$1$ in Theorem~\ref{t:BK1} is approached with a power savings: \.
LHS~$\ts = \ts 1+O(N^{-\eta})$ \. for some $\eta>0$.   We further remark
that a recent work of Rickards and Stange \cite{RS24}, shows that it is possible
to exhibit finitely generated semigroups $\G\subset \SL_2(\Z)$,
and base vectors $v_0,v_1\in\Z^2$, such that none of the elements
of the orbit \ts $\cO:=\<v_1\G,v_0\>$ are, e.g., perfect squares.
This disproves a conjecture by Bourgain and Kontorovich \cite[Conj.~1.11]{BK18},
and shows that there can be a Brauer--Manin type reciprocity obstruction,
meaning that a power savings asymptotic error is best possible in this generality.
The full local-global principle, that, say, $\cO$ contains \emph{every}
sufficiently large integer,  may well be true for our~$\G_A$ (for $A$ sufficiently large); this would imply the Diophantine Conjecture~\ref{conj:Zaremba-strong}.
Unfortunately, the general orbital circle method cannot distinguish
this setting from the similar setting where the statement is false. Therefore, these techniques
will not, without significant further ideas, be able to upgrade Theorem~\ref{thm:main2} to the full Conjecture~\ref{conj:main}.

\subsection{}\label{ss:finrem-Zaremba}
In Zaremba's Conjecture~\ref{conj:Zaremba}, it is known that $A=4$ is not enough
(take $u=6$ or~$54$).  Hensley (1996)
conjectured that one can take \ts $A=2$ \ts for \ts $d$ \ts large enough.
Zaremba's conjecture is known to hold for integers of the form \ts $2^m\ts 3^n$
and for sufficiently large powers of all primes, where the constant~$A$ can depend
on the prime, see \cite{Nid86,Shu24} for details.  We refer to \cite[$\S$6.2]{BPSZ14}
for an elegant presentation of the \ts $2^m$ \ts case.

\subsection{}\label{ss:Pollicott}
After this paper was written, Pollicott~\cite{Pol} computed the first 20 digits of
Hausdorff dimensions $\vt_A$ for
\[  \vt_{108} \ = \ 0.77474... \ , \qquad  \vt_{109} \ = \ 0.77490... \ , \qquad \vt_{\infty} \ = \ 0.79885...,     \]
of which we display only the first five digits here for brevity.
In particular this implies \. $\vt_{109} > (\sqrt{40}-4)/3 \approx 0.77485$\., so taking $A=109$ suffices 
for the proof of   Theorem~\ref{t:main}.

\vskip.6cm
{\small

\subsection*{Acknowledgements}
We are grateful to Noga Alon, Milan Haiman, Dmitry Krachun, Noah Kravitz,
Mark Pollicott, Peter Sarnak, Mehtaab Sawhney, Ilya Shkredov and
Wadim Zudilin for
interesting discussions and helpful remarks.  We thank Peter Sarnak for pointing
out the connection to spectra of operators discussed in~$\S$\ref{sec:Spec},
to Nikita Shulga for telling us about \cite{HanclTurek2023},
and to Oliver Jenkinson for emailing us~\cite{Jen04}.
We are especially grateful to Polina Vytnova for
patiently explaining to us the key ideas in \cite{PV22}
for rigorously computing Hausdorff dimensions.

SHC~was supported by NSF grant DMS-2246845.  AK~was supported by
NSF grant DMS-2302641, BSF grant 2020119 and a Simons Fellowship.
IP~was supported by NSF grant CCF-2302173.  This paper
was written when AK was visiting Princeton University and
IP was a member at the Institute of Advanced Study in Princeton,~NJ.
We are grateful for the hospitality.
}

\vskip1.1cm


{\footnotesize

\vskip.6cm
}

\end{document}